\documentclass[10pt]{amsart}
\usepackage{amsfonts}
\usepackage{indentfirst,amsthm,bm,amsmath}
\usepackage{hyperref}
\makeatletter
\date{}
\newtheorem{theorem}{Theorem}[section]
\newtheorem{corollary}{Corollary}[section]
\newtheorem{lemma}{Lemma}[section]

\newtheorem{remark}{Remark}[section]

\begin{document}

\title[Value distribution of $q$-differences in several variables]{Value distribution of $q$-differences of meromorphic functions in several complex variables}


\author[Tingbin Cao]{Tingbin Cao}
\address[Tingbin Cao]{Department of Mathematics, Nanchang University, Jiangxi 330031, P. R. China}
\email{tbcao@ncu.edu.cn}
\thanks{The first author was supported by the National Natural Science Foundation of China (\#11871260, \#11461042), the outstanding young talent assistance program of Jiangxi Province (\#20171BCB23002) in China.}

\author[Risto Korhonen]{Risto Korhonen}
\address[Risto Korhonen]{Department of Physics and Mathematics, University of Eastern Finland, P.O. Box 111, FI-80101 Joensuu, Finland}
\email{risto.korhonen@uef.fi}
\thanks{The second author was supported in part by the Academy of Finland grants (\#286877) and (\#268009).}

\subjclass[2010]{Primary 32H30; Secondary 30D35, 39A14}
\keywords{Logarithmic derivative lemma; Second main theorem; Partial $q$-difference equations; Picard type theorem; Tumura-Clunie theorem}
\date{}


\begin{abstract}
In this paper, we study $q$-difference analogues of several central results in value distribution theory of several complex variables such as  $q$-difference versions of the logarithmic derivative lemma, the second main theorem for hyperplanes and hypersurfaces, and a Picard type theorem. Moreover, the Tumura-Clunie theorem concerning partial $q$-difference polynomials is also obtained. Finally, we apply this theory to investigate the growth of meromorphic solutions of linear partial $q$-difference equations.
\end{abstract}

\maketitle


\section{Introduction}
At the beginning of the twentieth century, the foundations of the  general theory of the linear difference equations in one independent variable was built by N\"{o}rlund \cite{norlund}, Carmichael \cite{carmichael} and Pincherle \cite{pincherle}, and of the linear $q$-difference equations by Jackson \cite{jackson}, Mason \cite{mason-1914, mason} and Adams \cite{adams-1929}. Meanwhile, the general theory of difference equations in more than one independent variable, or of linear partial difference equations, was studied by Adams \cite{adams-1, adams-2, adams-3,adams-4}. For this background of linear partial $q$-difference equations, we refer to a review article of Adams \cite{adams-1931}.\par

Since 1925, when R. Nevanlinna \cite{nevanlinna-2} established the theory, the value distribution theory of meromorphic functions of one complex variable has found numerous applications in the theory of complex differential equations. The most striking result, in the Nevanlinna theory of value distribution, is the second main theorem, which is an inequality relating two leading quantities in the value distribution theory, namely, the characteristic function, which measures the rate of growth of a function or a map, and the counting function, which tells the size of the preimages of points or sets. Later on,  many forms of the second main theorem for holomorphic maps, as well as meromorphic maps, on various contexts were found. For example, in 1933, H. Cartan \cite{cartan} extended Nevanlinna's second main theorem for the case of holomorphic curves into complex projective spaces sharing hyperplanes in general position  (see also the 2004 review paper by Gundersen and Hayman \cite{gundersen-hayman} where Cartan's second main theorem is applied to study complex equations.) For the background of the Nevanlinna theory of one and several complex variables, we refer to \cite{hayman, yang, ru-1, noguchi-winkelmann}. It is well known that the logarithmic derivative lemma for meromorphic functions $f$ in one and several complex variables
$m(r, \frac{f'}{f})=O(\log (rT_{f}(r)))$ holds outside a possible set of finite linear measure. This lemma plays a key role in proving the celebrated forms of second main theorem and also has numerous applications in complex differential equations (see, e.g.,  \cite{laine} and references therein). \par

In recent years, in order to consider the growth of entire or meromorphic solutions of complex difference equations, the difference analogues of the second main theorem for meromorphic functions and maps were established. In 2006, Halburd and the second author \cite{halburd-korhonen-1} obtained a difference analogue of the second main theorem for meromorphic functions in the complex plane. Halburd, the second author and Tohge~\cite{halburd-korhonen-tohge-1}, and Wong, Law and Wong \cite{wong-law-wong-1}  obtained, independently, a difference analogue of the second main theorem for holomorphic curves into complex projective spaces intersecting with hyperplanes in general position. Recently, the second main theorem for the case of slowly moving periodic hyperplanes and slowly moving periodic hypersurfaces were obtained by the second author, Li and Tohge~\cite{korhonen-li-tohge}, and the first author and Nie \cite{cao-nie}, respectively. A tropical variant of the Cartan second main theorem for differences was recently found in \cite{korhonen-tohge}. The proofs of all these versions of the difference second main theorem rely on a difference analogue of logarithmic derivative lemma, which was originally discovered in two independent studies by Halburd and the second author \cite{halburd-korhonen-1}, and Chiang and Feng \cite{chiang-feng}, respectively. This lemma has also been one of the key tools in the study of meromorphic solutions of difference equations (see, e. g., \cite{chen-book} and references therein).\par

The main purpose of this paper is to establish several key results in the value distribution theory of $q$-differences in several complex variables, and then apply this theory  into the study of entire or meromorphic solutions of linear partial $q$-difference equations. Note that there is a difficulty (mentioned in \cite{cao}) to obtain $q$-difference analogue of the logarithmic derivative lemma for meromorphic functions in several complex variables if adopting the method due to Biancofiore and W. Stoll \cite{biancofiore-stoll}. In this paper, combining another method due to Stoll  \cite{stoll} and Fujimoto \cite{fujimoto-7} with the known method of the present authors \cite{korhonen-2012, cao-korhonen, cao-xu}, we firstly give a $q$-difference analogue of the logarithmic derivative lemma for meromorphic functions with zero order in several complex variables
\begin{equation*}
m\left(r, \frac{f(qz_{1}+c_{1}, \ldots, qz_{m}+c_{m} )}{f(z_{1}, \ldots, z_{m})}\right)=o(T_{f}(r))
\end{equation*}
on a set of logarithmic density one,  where $q\in\mathbb{C}\setminus\{0, 1\},$ $c=(c_{1}, \ldots, c_{m})\in\mathbb{C}^{m}$ (see in Section~\ref{qDifflemma_sec}). Then by this $q$-difference version of logarithmic derivative lemma, we obtain $q$-difference versions of the second main theorem for hyperplanes and hypersurfaces in Section~\ref{q_diff_hyperplane_sec}. As an application of our second main theorem, we get a $q$-difference Picard type theorem corresponding to Fujimoto and Green's result in classical value distribution theory in Section~\ref{picard_sec}. This implies that \emph{if a meromorphic function $f(z)$ on $\mathbb{C}^{m}$ and its shift $f(qz+c)$ share three distinct values (counting multiplicities) on $\mathbb{C}\cup\{\infty\},$ then $f(z)\equiv f(qz+c).$}  In Section~\ref{TCsec}, we discuss the $q$-difference counterpart of the Tumura-Clunie theorem in several complex variables concerning partial $q$-difference polynomials
\begin{eqnarray*}
G(z, f)=\sum_{\lambda\in J} b_{\lambda}(z)\prod_{j=1}^{\tau_{\lambda}}f(q_{\lambda, j}z+c_{\lambda,j})^{\mu_{\lambda,j}}.
\end{eqnarray*} Finally, in Section~\ref{LEsec}, we apply the $q$-difference value distribution theory to study the growth of meromorphic solutions of the linear partial $q$-difference equation of the form \begin{eqnarray*}
\sum_{j=0}^{n}A_{j}(z_{1}, \ldots, z_{m})f(q_{j}z_{1}+c_{j,1}, \ldots, q_{j}z_{m}+c_{j, m})=F(z_{1}, \ldots, z_{m})
\end{eqnarray*} with meromorphic coefficients for $m$ independent complex variables $z_{1},$ $\ldots,$ $z_{m},$ where $q_{j}\in\mathbb{C}\setminus\{0, 1\},$ $c_{j}=(c_{j,1}, \ldots, c_{j, m})\in\mathbb{C}^{m}$ for all $j=1, \ldots, n.$ Necessary notation has been given in the following section.

\section{Preliminaries}


\noindent\textbf{2.1} Set $\|z\|=(|z_{1}|^{2}+\cdots+|z_{m}|^{2})^{1/2}$ for $z=(z_{1},\cdots,z_{m})\in\mathbb{C}^{m},$  for $r>0,$ define
\begin{equation*}
B_{m}(r):={\{z\in\mathbb{C}^{m}:\|z\|\leq r}\},\quad
S_{m}(r):={\{z\in\mathbb{C}^{m}:\|z\|=r}\}. \end{equation*}
Let $d=\partial+\overline{\partial}, \quad d^{c}=(4\pi\sqrt{-1})^{-1}(\partial-\overline{\partial}).$ Thus $dd^{c}=\frac{\sqrt{-1}}{2\pi}\partial\overline{\partial}.$ Write
\begin{equation*}
\sigma_{m}(z):=(dd^{c}\log\|z\|^{2})^{m-1},\quad
 \eta_{m}(z):=d^{c}\log\|z\|^{2}\wedge \sigma_{m}(z)\end{equation*}
  for $z\in\mathbb{C}^{m}\setminus{\{0}\}.$\par

For a divisor $\nu$ on $\mathbb{C}^{m}$ we define the following counting function of $\nu$ by

\begin{equation*}
N(r, \nu)=\int_{1}^{r}\frac{n(t)}{t^{2m-1}}dt\quad
(1<r<\infty),\end{equation*} where \begin{eqnarray*} n(t)=\left\{
                                    \begin{array}{ll}
                                      \int_{|\nu|\cap B(t)}\nu(z)\sigma_{m}(z), & \hbox{if $m\geq 2;$} \\
                                      \sum_{|z|\leq t}\nu(z), & \hbox{if $m=1.$}
                                    \end{array}
                                  \right.
\end{eqnarray*}

Let $\varphi(\not\equiv 0)$ be an entire function on $\mathbb{C}^{m}.$ For $a\in\mathbb{C}^{m},$ we write $\varphi(z)=\sum_{i=0}^{\infty}P_{i}(z-a),$ where the term $P_{i}$ is a homogeneous polynomial of degree~$i.$ We denote the zero-multiplicity of $\varphi$ at $a$ by $\nu_{\varphi}(a)=\min{\{i:P_{i}\not\equiv 0}\}.$ Thus we can define a divisor $\nu_{\varphi}$ such that $\nu_{\varphi}(z)$ equals the zero multiplicity of $\varphi$ at $z$ in the sense of \cite[Definition 2.1]{fujimoto-2} whenever $z$ is a regular point of an analytic set  $|\nu_{\varphi}|:=\overline{\{z\in \mathbb{C}^{m}: \nu_{\varphi}(z)\neq 0\}}.$ \par

Let $h$ be a nonzero meromorphic function on $\mathbb{C}^{m}$ with $h=\frac{h_{0}}{h_{1}},$ where $h_{0}$ and $h_{1}$ are two entire functions on $\mathbb{C}^{n}$ such that $\dim(h_{0}^{-1}(0)\cap h_{1}^{-1}(0))\leq m-2.$ We define $\nu_{h}^{0}:=\nu_{h_{0}}$ and $\nu_{h}^{\infty}:=\nu_{h_{1}}.$\par

For a meromorphic function $h$ on $\mathbb{C}^{m},$ we usually write $N(r, \frac{1}{h}):=N(r, \nu_{h}^{0})$ and $N(r, h):=N(r, \nu_{h}^{\infty}).$  The Jensen's Formula is given as
$$N(r, \frac{1}{h})-N(r, h)=\int_{S_{m}(r)}\log|h|\eta_{m}(z)-\int_{S_{m}(1)}\log|h|\eta_{m}(z).$$
The proximity function of $h$ is defined by $$m(r, h)=\int_{S_{m}(r)}\log^{+}|h(z)|\eta_{m}(z),$$ where $\log^{+}x:=\max\{\log x, 0\}$ for any $x>0.$ \par


\medskip

\noindent\textbf{2.2} A meromorphic mapping $f:\mathbb{C}^{m}\rightarrow\mathbb{P}^{n}(\mathbb{C})$ is a holomorphic mapping from $U$ into $\mathbb{P}^{n}(\mathbb{C}),$ where $U$  can be chosen so that $K_f\equiv \mathbb{C}^{m}\setminus U$ is an analytic subvariety of $\mathbb{C}^{m}$ of codimension at least $2.$ Furthermore $f$ can be represented by a holomorphic mapping of $\mathbb{C}^{m}$ to $\mathbb{C}^{n+1}$ such that $$K_{f}=\{z\in\mathbb{C}^{m}: f_{0}(z)=\cdots=f_{n}(z)=0\},$$ where $f_{0}, \ldots, f_{n}$ are holomorphic functions on $\mathbb{C}^{m}.$ We say that $f=[f_{0}, \ldots, f_{n}]$ is a reduced representation of $f$ (the only factors common to $f_{0}, \ldots, f_{n}$ are units).  If $g=hf$ for $h$ any quotient of holomorphic functions on $\mathbb{C}^{m},$ then $g$ will be called a representation of $f$ (e.g. reduced if and only if $h$ is holomorphic and a unit). Set $\|f\|=(\sum_{j=0}^{n}|f_{j}|^{2})^{\frac{1}{2}}.$ The Nevanlinna-Cartan's characteristic function for a meromorphic mapping $f$ is defined by
\begin{eqnarray*}
T_{f}(r)
&=&\int_{S_{m}(r)}\log\|f\|\eta_{m}(z)-
 \int_{S_{m}(1)}\log\|f\|\eta_{m}(z)\\
 &=&\int_{S_{m}(r)}\log\max\{|f_{0}|, \ldots, |f_{n}|\}\eta_{m}(z)+O(1)\quad(r>r_{0}>1).\end{eqnarray*}
Note that $T_{f}(r)$ is independent of the choice of the reduced representation of $f.$ The order of $f$ is defined by
$$\zeta(f):=\limsup_{r\rightarrow\infty}\frac{\log^{+} T_{f}(r)}{\log r}.$$ \par

For $q=(q_{1}, q_{2}, \ldots, q_{m})\in\mathbb{C}^{m}\setminus\{0\},$ we denote by $\mathcal{M}$ the set of all meromorphic functions over $\mathbb{C}^{m},$ by $\mathcal{P}_q$ the set of all meromorphic functions $h\in\mathcal{M}$ satisfying $h(q_{1}z_{1}+c_{1}, \ldots, q_{m}z_{m}+c_{m})\equiv h(z_{1}, \ldots, z_{m})$ for some fixed value $c=(c_{1}, \ldots, c_{m})\in\mathbb{C}^{m},$ and by $\mathcal{P}_{q}^{0}$ the set of all meromorphic functions in $\mathcal{P}_{q}$ and having zero order. Obviously, then we have the inclusions $\mathcal{M}\supset \mathcal{P}_q \supset \mathcal{P}_{q}^{0} \supset \mathbb{C}.$ As a matter of fact,  in  the special case of $c=(0,\ldots, 0)$ in $\mathbb{C}^{m},$ if, for any $j\in\{1,\ldots,m\}$, either $|q_j|\neq 1$, or $|q_j|=1$ and the argument of $q_j$ is not rational, then we have $\mathcal{P}_q=\mathcal{P}_q^0=\mathbb{C}.$  \par

We say that a meromorphic mapping $f$ from $\mathbb{C}^{m}$ into $\mathbb{P}^{n}(\mathbb{C})$ with a reduced representation $[f_{0},\ldots, f_{n}]$ is algebraically (linearly) nondegenerate over $\mathcal{P}^{0}_{q}$  if the entire functions $f_{0},$ $\ldots,$ $f_{n}$ are algebraically (linearly) independent over $\mathcal{P}^{0}_{q},$ and say that $f$ is algebraically (linearly) nondegenerate over $\mathbb{C}$  if the entire functions $f_{0},$ $\ldots,$ $f_{n}$ are algebraically (linearly) independent over $\mathbb{C}.$ \par


\medskip

\noindent\textbf{2.3} A hypersurface $Q$ with degree $d$ in $\mathbb{P}^{n}(\mathbb{C})$ is given by
$$Q=\{[x_0:\cdots:x_n]:\sum_{I\in\mathcal{J}_d}a_{I}x^I=0\},$$ where $\mathcal{J}_d=\{(i_0,\cdots, i_n)\in \mathbb{N}_0^{n+1}:i_0+\cdots+i_n=d\},$ $I=(i_0, \cdots, i_n)\in \mathcal{J}_{d},$ $x^I=x_0^{i_0}\cdots x_n^{i_n}$ and $(x_0:\cdots:x_n)$ are homogeneous coordinates of $\mathbb{P}^n(\mathbb{C})$. Denote by $D$ the homogeneous polynomial associated with the hypersurface $Q.$ If $d$ is $1$, then the hypersurface reduces to one hyperplane, denoted by $H,$ as  $$H=\{[x_0:\cdots:x_n]:a_{0}x_{0}+\cdots+a_{n}x_{n}=0\}.$$ Set
$$Q(f(z)):=D\circ f(z)=\sum_{I\in\mathcal{J}_d}a_{I}f^I,$$ where $f^I=f_0^{i_0}\cdots f_n^{i_n}.$ We recall the proximity function of $f$ intersecting $Q$ defined as $$m_{f}(r, Q)=\int_{S_{m}(r)}\log\frac{\|f(z)\|^{d}\|a\|^{d}}{|Q(f(z))|}\eta_{m}(z),$$ where $\|a\|=(\sum_{I\in \mathcal{J}_{d}}|a_{I}|^{2})^{1/2}.$ Throughout this paper, we usually assume that $f(\mathbb{C}^{m})\not\subset Q$ without a special statement. Then we have the first main theorem as follows:
$$m_{f}(r, Q)+N\left(r, \frac{1}{Q(f)}\right)=dT_{f}(r)+O(1).$$\par

Now let $\{Q_i\}_i^p$ be hypersurfaces of $\mathbb{P}^n(\mathbb{C}).$ We say that the family of the hypersurfaces $\{Q_j\}_{j=1}^p$ are in general position in $\mathbb{P}^n(\mathbb{C})$ if for any subset $R\subset Q$ with the cardinality $\sharp R=n+1,$  we have
$$\bigcap_{j\in R}Q_j=\emptyset.$$ That is, any $n+1$ homogeneous polynomials (forms) of $\{D_j(z)\}_{j=1}^{p}$ associated with the hypersurfaces $\{Q_{j}\}_{j=1}^{p}$ are linearly independent over $\mathbb{C}.$ \par


\medskip

\noindent\textbf{2.4} Let $f$ be a meromorphic mapping from $\mathbb{C}^{m}$ into $\mathbb{P}^{n}(\mathbb{C}).$ In what follows, $0=(0,\ldots,0)\in\mathbb{C}^{m},$ $1=(1,\ldots,1)\in\mathbb{C}^{m}.$ For $q=(q_{1}, \ldots, q_{m})$ and  $z=(z_{1}, \ldots, z_{m}),$ we write $q+z=(q_{1}+z_{1},\ldots,q_{m}+z_{m}),$ $q\cdot z=(q_{1}z_{1},\ldots,q_{m}z_{m}).$
Denote the $q$-difference operator by
\begin{eqnarray*}\Delta_{q}f&:=&f(q\cdot z+c)-f(z)\\&=&f(q_{1}z_{1}+c_{1}, q_{2}z_{2}+c_{2}, \ldots, q_{m}z_{m}+c_{m})-f(z_{1}, z_{2}, \ldots, z_{m}).\end{eqnarray*}

For $q\in\mathbb{C}^{m}\setminus\{0\},$ and a meromorphic mapping $f:\mathbb{C}^{m}\rightarrow\mathbb{P}^{n}(\mathbb{C})$ with a reduced representation $f=[f_{0}, \ldots, f_{n}],$ we use the short notations
$$\overline{f}^{[0]}=f:=f(z), \,\,\,\overline{f}^{[1]}=\overline{f}:=f(q\cdot z+c),\,\,\, \overline{f}^{[k]}:=\overline{\overline{f}^{[k-1]}}.$$ Then, analogously to the definitions of the Wronskian and the Casoratian determinants, the $q$-Casorati determinant of $f$ is defined by
$$ C(f)=C(f_0,\cdots,f_n)=\left|\begin{array}{cccc}
f_0 & f_1 & \cdots & f_n\\
\overline{f_0} & \overline{f_1} & \cdots & \overline{f_n}\\
\vdots & \vdots & \vdots & \vdots\\
\overline{f_0}^{[n]} & \overline{f_1}^{[n]} & \cdots & \overline{f_n}^{[n]}
\end{array}\right|_{(n+1)\times(n+1)}.$$ Given a real positive integer $d,$ $\mathcal{J}_d=\{(i_0, \cdots, i_n)\in \mathbb{N}_0^{n+1}: i_0+\cdots+ i_n=d\}.$ For any $I_{j}=(i_{j0}, \ldots, i_{jn})\in\mathcal{J}_d$ where $j\in\{1,\ldots, M\},$ we set $f^{I_{j}}=f_0^{i_{j0}}\cdots f_n^{i_{jn}}.$ Then $q$-Casorati determinant of $f$ is given as
\begin{eqnarray*} \tilde{C}(f)=C(f^{I_{1}}, \ldots, f^{I_{M}})=\left|\begin{array}{cccc}
f^{I_{1}} & f^{I_{2}} & \cdots & f^{I_{M}}\\
\overline{f^{I_{1}}} & \overline{f^{I_{2}}} & \cdots & \overline{f^{I_{M}}}\\
\vdots & \vdots & \vdots & \vdots\\
\overline{f^{I_{1}}}^{[M-1]} & \overline{f^{I_{2}}}^{[M-1]} & \cdots & \overline{f^{I_{M}}}^{[M-1]}
\end{array}\right|_{M\times M}.\end{eqnarray*} Clearly, when $d=1$ and $M=n+1,$ we have $|\tilde{C}(f)|=|C(f)|.$ Moreover, one can rearrange the order of $I_{1}, \ldots, I_{M}$ such that $\tilde{C}(f)=C(f)$ whenever $d=1$ and $M=n+1.$\par

\section{$q$-Difference analogue of the logarithmic derivative lemma in several complex variables}\label{qDifflemma_sec}
The original version of the logarithmic derivative lemma in one complex variable plays a key role in Nevanlinna theory for meromorphic functions in the complex plane, and is widely used in value distribution
of meromorphic functions, differential equations in the
complex plane, and elsewhere. The first generalization of
the logarithmic derivative lemma from one variable to several complex
variables was given by Vitter \cite{vitter}, another proof was given by
Biancofiore and Stoll \cite{biancofiore-stoll}.  \par

In \cite[Theorem 1.1]{barnett-halburd-korhonen-morgan}, a $q$-difference analogue of  the logarithmic derivative lemma for meromorphic functions on the complex plane $\mathbb{C}$ was obtained. Note that the assumption of $f$ with zero order is sharp.\par

\begin{theorem}[\cite{barnett-halburd-korhonen-morgan}] \label{T-3.1} Let $f$ be a non-constant zero-order meromorphic function on $\mathbb{C},$ and $q\in \mathbb{C}\setminus\{0, 1\}.$ Then \begin{equation*}
m\left(r, \frac{f(qz)}{f(z)}\right)=o(T_{f}(r))
\end{equation*} on a set of logarithmic density $1.$
 \end{theorem}

It is natural to consider how to extend Theorem \ref{T-3.1} to several complex variables. In \cite{cao}, the first author of this paper attempted to use the Biancofiore-Stoll method and found that there is one big technical difficulty in extending Theorem \ref{T-3.1} to the case of several complex variables. Meanwhile, Wen \cite{wenzhitao} also used the same method and stated that Theorem \ref{T-3.1} is still true for a meromorphic function on $\mathbb{C}^{m}$ and a given $q\in\mathbb{C}^{m}\setminus\{0\}.$ Unfortunately, there is a gap in \cite{wenzhitao}, which has to do with the same technical problem as the first author of this paper met in \cite{cao}. In fact, from the proof of \cite[Lemma 5]{wenzhitao} one can see that the term $(\frac{R}{r})^{2n-2}$ is missing on the right hand side of the inequality in the statement of \cite[Lemma 5]{wenzhitao},  which means that the equality (2.3) in \cite{wenzhitao} cannot be obtained by taking $R=kr$ and making use of \cite[Lemma 5 and Lemmas 6-8]{wenzhitao} in the proof of \cite[Theorem 1]{wenzhitao}. 
From the following result on the $q$-difference analogue of the logarithmic derivative lemma we conclude that \cite[Theorem~1]{wenzhitao} and its applications remain valid under the additional assumption of $q=(\tilde{q}, \ldots, \tilde{q})\in\mathbb{C}^{m}\setminus\{0\}.$\par

Here, we will adopt the method due to Stoll \cite{stoll} and Fujimoto \cite{fujimoto-7} to obtain a weak $q$-difference analogue of the logarithmic derivative lemma for meromorphic functions of several complex variables and a specially given $q=(\tilde{q}, \ldots, \tilde{q})\in\mathbb{C}^{m}\setminus\{0, 1\}.$ This result generalizes Theorem \ref{T-3.1}, but it is still open whether the result remains true for any given $q\in\mathbb{C}^{m}\setminus\{0\}.$ Since for $q=1\in\mathbb{C}^{m}$ it is already studied by \cite{korhonen-2012, cao-korhonen, cao-xu}, we here only assume $q\in\mathbb{C}^{m}\setminus\{0, 1\}.$\par

\begin{theorem}\label{q-LDL}
Let $f$ be a nonconstant zero-order meromorphic function on $\mathbb{C}^{m},$ and let $q=(\tilde{q}, \ldots, \tilde{q})\in\mathbb{C}^{m}\setminus\{0,  1\},$ $c=(c_{1}, \ldots, c_{m})\in\mathbb{C}^{m}.$ Then
\begin{eqnarray*}
m\left(r, \frac{f(q\cdot z+c)}{f(z)}\right)&=&\int_{S_{m}(r)}\log^{+}\left|\frac{f(\tilde{q}z_{1}+c_{1}, \ldots, \tilde{q}z_{m}+c_{m})}{f(z_{1}, \ldots, z_{m})}\right|\eta_{m}(z)\\&=&o\left(T_{f}(r)\right)
\end{eqnarray*} for all $r=\|z\|$ on a set of logarithmic density one.
\end{theorem}

\begin{proof}
Let $E_{1}$ be the set of all points $\xi\in S_{m}(1)$ such that $\{z=u\xi: |u|<+\infty\}\subset K_{f}$ which is of measure zero in $S_{m}(1).$ For any $\xi\in S_{m}(1)\setminus E_{1},$ considering the meromorphic function $f^{\xi}(u):=f(\xi u)$ of $\mathbb{C}^{1},$ we have
 \begin{eqnarray*} T_{f^{\xi}}(r)=\frac{1}{2\pi}\int_{0}^{2\pi}\log |f^{\xi}(re^{\sqrt{-1}\theta})|d\theta-\log|f(0)|,\end{eqnarray*} and thus by \cite[Lemmas 1.1--1.2]{stoll} it follows (see also \cite[pp.~33--34]{fujimoto-7})
\begin{equation}\label{Eqq-3.1}T_{f}(r)=\int_{S_{m}(1)}T_{f^{\xi}}(r)\eta_{m}(z).\end{equation}\par

Recall that the proximity function of the meromorphic function $f$ on $\mathbb{C}^{m}$ is defined by
\begin{equation*}
m(r, f)=\int_{S_{m}(r)}\log^{+}\left|f(z)\right|\eta_{m}(z).
\end{equation*}
For any $\xi\in S_{m}(1)\setminus E_{1},$ considering the meromorphic function $f^{\xi}(u):=f(\xi u)$ of $\mathbb{C}^{1},$ we have
$$m(r, f^{\xi})=\frac{1}{2\pi}\int_{0}^{2\pi}\log^{+}|f^{\xi}(re^{\sqrt{-1}\theta})|d\theta.$$ Then by \cite[Lemmas 1.1--1.2]{stoll}, we also get
\begin{eqnarray}\label{Eqq-3.2}m(r, f)=\int_{S_{m}(1)}m(r,¡¡f^{\xi})\eta_{m}(z).\end{eqnarray} \par

Since $q=(\tilde{q}, \ldots, \tilde{q})\in\mathbb{C}^{m}\setminus\{0\},$  $\tilde{q}\in\mathbb{C}^{1}\setminus\{0\}.$ For any $\xi\in S_{m}(1)\setminus E_{1},$ considering the meromorphic function $f^{\xi}(u):=f(\xi u)$ of $\mathbb{C}^{1},$ we get from \eqref{Eqq-3.2} that
\begin{eqnarray*} m\left(r, \frac{f(q\cdot z)}{f(z)}\right)&=&
  \int_{S_{m}(r)}\log^{+}\left|\frac{f(q\cdot z)}{f(z)}\right|\eta_{m}(z)
  \\&=&\int_{S_{m}(1)}\left(\frac{1}{2\pi}\int_{0}^{2\pi}\log^{+}\left|\frac{f^{\xi}(|\tilde{q}u|e^{\sqrt{-1}\theta})}{f^{\xi}(|u|e^{\sqrt{-1}\theta})}\right|d\theta\right)\eta_{m}(z),\end{eqnarray*}
where we denote $z=u\xi$ for any $\xi\in S_{m}(1).$ By \cite[Lemma 5.1]{barnett-halburd-korhonen-morgan}, we get that for all $r>0,$ $M>\max\{1, \|q\|=|\tilde{q}|\},$\begin{eqnarray*}
m\left(r,\frac{f^{\xi}(|\tilde{q}u|e^{\sqrt{-1}\theta})}{f^{\xi}(|u|e^{\sqrt{-1}\theta})}\right)
&=&\frac{1}{2\pi}\int_{0}^{2\pi}\log^{+}\left|\frac{f^{\xi}(|\tilde{q}u|e^{\sqrt{-1}\theta})}{f^{\xi}(|u|e^{\sqrt{-1}\theta})}\right|d\theta\\&\leq& \frac{O(1)}{M}\left(T_{f^{\xi}}(Mr)+\log^{+}\frac{1}{|f^{\xi}(0)|}\right),\end{eqnarray*}
where $r=|u|=\|u\xi\|=\|z\|.$ Therefore, together with \eqref{Eqq-3.1}, it follows from the two inequalities above that
\begin{eqnarray*} m\left(r, \frac{f(q\cdot z)}{f(z)}\right)&=&
  \int_{S_{m}(r)}\log^{+}\left|\frac{f(q\cdot z)}{f(z)}\right|\eta_{m}(z)
  \\&\leq&\int_{S_{m}(1)}\left(\frac{O(1)}{M}\left(T_{f^{\xi}}(Mr)+\log^{+}\frac{1}{|f^{\xi}(0)|}\right)\right)\eta_{m}(z)\\
  &=&
\frac{O(1)}{M}\int_{S_{m}(1)}T_{f^{\xi}}(Mr)\eta_{m}(z)+O(1),\end{eqnarray*}
namely, \begin{eqnarray}\label{Eqq-3.3} m\left(r, \frac{f(q\cdot z)}{f(z)}\right)\leq \frac{O(1)}{M} T_{f}(Mr)+O(1).\end{eqnarray}

The following part of the proof is dealt with similarly as in \cite{barnett-halburd-korhonen-morgan}. By choosing $M:=2^{n}$ and by applying \cite[Lemma 4]{hayman}, we get $T_{f}(Mr)\leq 2T_{f}(r)$ on a set of logarithmic density one. Hence by \cite[Lemma 5.3]{barnett-halburd-korhonen-morgan}  we get from \eqref{Eqq-3.3} that
\begin{eqnarray*}m\left(r, \frac{f(q\cdot z)}{f(z)}\right)=o\left(T_{f}(r)\right)\end{eqnarray*} for  all $r$ on a set of logarithmic density one.\par

According to \cite[Theorem 3.1]{korhonen-2012} (see also \cite{cao-xu} for an improvement), we have
\begin{eqnarray*} m\left(r, \frac{f(z+c)}{f(z)}\right)\leq o(T_{f}(r))\end{eqnarray*} for all $r$ outside of a possible exceptional set of finite logarithmic measure (which is of logarithmic density zero).

Hence, we get that \begin{eqnarray*} m\left(r, \frac{f(q\cdot z+c)}{f(z)}\right)\leq m\left(r, \frac{f(q\cdot z+c)}{f(q\cdot z)}\right)+m\left(r, \frac{f(q\cdot z)}{f(z)}\right)=o(T_{f}(r))\end{eqnarray*} for  all $r$ on a set of logarithmic density one.

\end{proof}

\begin{remark} Note that $$\frac{f(z)}{f(q\cdot z+c)}=\frac{f(\frac{1}{q_{1}}(q_{1}z_{1}+c_{1})-\frac{c_{1}}{q_{1}}, \ldots, \frac{1}{q_{m}}(q_{m}z_{m}+c_{m})-\frac{c_{m}}{q_{m}})}{f(q_{1}z_{1}+c_{1}, \ldots, q_{m}z_{m}+c_{m})}$$ and $$\frac{\overline{f}^{[k]}}{f}=\frac{\overline{f}^{[k]}}{\overline{f}^{[k-1]}}\cdot\frac{\overline{f}^{[k-1]}}{\overline{f}^{[k-2]}}\cdots\frac{\overline{f}}{f}\, \,(k\in\mathbb{N}).$$ Thus we get immediately from Theorem \ref{q-LDL} that for $q=(\tilde{q}, \ldots, \tilde{q})\in\mathbb{C}^{m}\setminus\{0, 1\},$
\begin{equation*}
\int_{S_{m}(r)}\log^{+}\left|\frac{f(z)}{f(q\cdot z+c)}\right|\eta_{m}(z)=o\left(T_{f}(r)\right)
\end{equation*}
and
\begin{equation*}
 \int_{S_{m}(r)}\log^{+}\left|\frac{\overline{f}^{[k]}}{f}\right|\eta_{m}(z)+\int_{S_{m}(r)}\log^{+}\left|\frac{f}{\overline{f}^{[k]}}\right|\eta_{m}(z)=o\left(T_{f}(r)\right)
\end{equation*}
for  all $r$ on a set of logarithmic density one.\end{remark}

\section{$q$-difference analogue of the second main theorem for hyperplanes}\label{q_diff_hyperplane_sec}
In 2007, the second main theorem for meromorphic functions of one variable was obtained by Barnett, Halburd, the second author and Morgan \cite{barnett-halburd-korhonen-morgan}.\par

\begin{theorem}[\cite{barnett-halburd-korhonen-morgan}]\label{T-4A} Let $f$ be a non-constant zero-order meromorphic function on $\mathbb{C},$ let $\Delta^{0}_{q}f:=f(qz)-f(z)\not\equiv 0$ and $q\in \mathbb{C}\setminus\{0, 1\},$ and let $a_{1},a_{2}, \ldots, a_{p}\in\mathbb{C},$ $p\geq 2,$ be distinct points. Then
\begin{equation*}m(r,f)+\sum_{k=1}^{p}m\left(r, \frac{1}{f-a_{k}}\right)\leq 2T_{f}(r)-N_{\mbox{pair}}(r,f)+o(T_{f}(r))\end{equation*} on a set with a logarithmic density one, where \begin{equation*}N_{\mbox{pair}}(r,f):=2N(r,f)-N(r,\Delta^{0}_{q}f)+N\left(r,\frac{1}{\Delta^{0}_{q}f}\right).
\end{equation*}
 \end{theorem}

In \cite{cao}, the first author extended Theorem \ref{T-4A} to the case for meromorphic functions of several complex variables and a given number $q=(\tilde{q}, \ldots, \tilde{q})\in\mathbb{C}^{m}\setminus\{0\},$ by directly applying the method of Stoll and Fujimoto (the standard process of averaging over the complex lines in $\mathbb{C}^{m}$) to the second main theorem of one variable, and obtained the following result.\par

\begin{theorem}[\cite{cao}]\label{T-4B} Let $f$ be a nonconstant meromorphic function of order zero on $\mathbb{C}^{m}$ such that $\Delta_{q}f\not\equiv 0,$ where $q=(\tilde{q}, \ldots, \tilde{q})\in\mathbb{C}^{m}\setminus\{0, 1\}$ and $c=(0, \ldots, 0).$ For any $p$ distinct values $a_{1}, a_{2}, \ldots, a_{p}\in\mathbb{P}(\mathbb{C})$,
\begin{eqnarray*}(p-2)T_{f}(r)\leq \sum_{j=1}^{p}N\left(r, \frac{1}{f-a_{j}}\right)-N\left(r, \frac{1}{C(f_{0}, f_{1})}\right)+o\left(T_{f}(r)\right)
\end{eqnarray*}on a set with a logarithmic density one, where $C(f_{0}, f_{1})$ is the $q$-Casorati determinant of $f$ with a reduced representation $f=[f_{0}, f_{1}].$
 \end{theorem}

It is known that holomorphic functions $g_{0}, \ldots, g_{n}$ on $\mathbb{C}^{m}$ are linearly dependent over $\mathbb{C}^{m}$ if and only if their Wronskian determinant $W(g_{0},\ldots, g_{n})$ vanishes identically \cite[Proposition~4.5]{fujimoto-3}. By the definition of the characteristic function and using a similar discussion as in \cite[Theorem 1.6,~p.~47]{goldberg-ostrovskii},  one can easily get that for any meromorphic function $h$ on $\mathbb{C}^{m},$  $q\in \mathbb{C}^{m}\setminus\{0\}$ and $c\in\mathbb{C}^{m},$
$$T_{h(q\cdot z+c)}(r)=O\left(T_{h(z)}(\|q\|r+\|c\|)\right).$$
Then considering this fact and making use of almost the same discussion as in \cite[Lemma 3.2]{halburd-korhonen-tohge-1}, we can get the following Lemma~\ref{L-4.1}~(ii). To prove (i) it is just not necessary to consider the growth of $f$ in the proof of (ii).  It is not difficult to get (iii) from (ii) for a meromorphic mapping $g=(f^{I_{1}}, \ldots, f^{I_{M}}): \mathbb{C}^{m}\rightarrow\mathbb{P}^{M-1}.$ We omit the details. \par

\begin{lemma}\label{L-4.1} (i) Let $q\in\mathbb{C}^{m}\setminus\{0, 1\}$ and $c\in\mathbb{C}^{m}.$
A meromorphic mapping $f:\mathbb{C}^{m}\rightarrow\mathbb{P}^{n}(\mathbb{C})$ with a reduced representation $[f_{0}, \ldots, f_{n}]$ satisfies $C(f_{0}, \ldots, f_{n})\not\equiv 0$ if and only if $f$ is linearly nondegenerate over the field $\mathcal{P}_{q}.$\par

(ii) Let $q\in\mathbb{C}^{m}\setminus\{0, 1\}$ and $c\in\mathbb{C}^{m}.$
If a meromorphic mapping $f:\mathbb{C}^{m}\rightarrow\mathbb{P}^{n}(\mathbb{C})$ with a reduced representation $[f_{0}, \ldots, f_{n}]$ satisfies  $\zeta(f)=0,$ then $C(f_{0}, \ldots, f_{n})\not\equiv 0$ if and only if $f$ is linearly nondegenerate over the field $\mathcal{P}_{q}^{0}(\subset\mathcal{P}_{q}).$\par

(iii) Let $q\in\mathbb{C}^{m}\setminus\{0, 1\}$ and $c\in\mathbb{C}^{m}.$ For a positive integer $M,$ set $\mathcal{J}_{\alpha}=\{(i_0, \cdots, i_n)\in \mathbb{N}_0^{n+1}: i_0+\cdots+ i_n=\alpha\},$ $I_{j}\in \mathcal{J}_{\alpha}$ for all $j\in\{1, \ldots, M\}.$ Then the meromorphic map $f=[f_0,\cdots, f_n]: \mathbb{C}^{m}\rightarrow\mathbb{P}^{n}(\mathbb{C})$ with zero order satisfies $\tilde{C}(f)=C(f^{I_{1}},\ldots, f^{I_{M}})\not\equiv 0$ if and only if the entire functions $f_0, \ldots, f_n$ are algebraically nondegenerate over the field $\mathcal{P}_{q}^{0}.$
\end{lemma}

The following result is a general version of the second main theorem for a meromorphic mapping $f$ into a complex projective space intersecting hyperplanes, with the ramification term in terms of the $q$-Casorati determinant, but not the Wronskian determinant like in \cite{ru-1}. Note that here we do not need a growth condition for $f$. \par

\begin{theorem}\label{GSMT} Let $q\in\mathbb{C}^{m}\setminus\{0, 1\},$ $c\in\mathbb{C}^{m},$ and $f=[f_{0}, \ldots, f_{n}]:\mathbb{C}^{m}\rightarrow\mathbb{P}^{n}(\mathbb{C})$ be a linearly nondegenerate meromorphic mapping over the field $\mathcal{P}_{q}.$ Let $H_{1}, \ldots, H_{p}$ (defining coefficient vectors $a_{1}, \ldots, a_{p}$ respectively) be arbitrary hyperplanes in $\mathbb{P}^{n}(\mathbb{C}).$ Let $T$ be the set of all injective maps $\mu: \{0, 1, \ldots, n\}\rightarrow \{1, \ldots, p\}$ such that $a_{\mu(0)}, \ldots, a_{\mu(n)}$ are linearly independent. Denote by $h_{\mu(l)}=<f(z), a_{\mu(l)}>$ and $g_{\mu(l)}=\frac{h_{\mu(l)}}{h_{\mu(0)}}$ for all $0\leq l\leq n.$ Then we have
\begin{eqnarray*}
&&\int_{S_{m}(r)}\max_{K}\sum_{k\in K}\lambda_{H_{k}}(f(z))\eta_{m}(z)\\
&\leq&\int_{S_{m}(r)}\max_{\mu\in T}\sum_{j=0}^{n}\log\frac{\|f(z)\|\|a_{\mu(j)}\|}{|<f(z), a_{\mu(j)}>|}\eta_{m}(z)\\
&\leq&(n+1)T_{f}(r)-N(r,\frac{1}{C(f_{0}, \ldots, f_{n})})+O(1)\\
&&+\sum_{\mu\in T}\sum_{i_{1}+\ldots+i_{n}\leq\frac{n(n+1)}{2}}
\sum_{l=1}^{n}\int_{S_{m}(r)}\log^{+}\frac{|\overline{g}_{\mu(l)}^{[i_{l}]}|}{|g_{\mu(l)}|}\eta_{m}(z)\\
&&+\sum_{\mu\in T}\sum_{l=1}^{n}\int_{S_{m}(r)}\log^{+}
\frac{|g_{\mu(l)}|}{|\overline{g}_{\mu(l)}^{[l]}|}\eta_{m}(z)+\sum_{\mu\in T}\sum_{j=0}^{n}\left(N\bigg(r, \frac{1}{\overline{h}_{\mu(j)}^{[j]}}\bigg)-N\left(r, \frac{1}{h_{\mu(j)}}\right)\right),
\end{eqnarray*}where $C(f_{0}, \ldots, f_{n})$ is the $q$-Casorati determinant of $f,$ and the maximum is taken over all subsets $K$ of $\{1, \ldots, p\}$ such that $a_{j}$ $(j\in K)$ are linearly independent.
\end{theorem}

\begingroup
\allowdisplaybreaks

\begin{proof} The basic idea comes from \cite[Theorem 3.1.3]{ru-1}. The difference is that we consider the Casorati determinant here.
Let $H_{1}, \ldots, H_{p}$ be the given hyperplanes with coefficient vectors $a_{1}, \ldots, a_{p}$ in $\mathbb{C}^{n+1}.$ Denote by $K\subset\{1, \ldots, p\}$ such that $a_{k} (k\in K)$ are linearly independent. Without loss of generality, we may assume that $p\geq n+1$ and $\sharp K=n+1.$ Let $T$ be the set of all injective maps $\mu: \{0, 1, \ldots, n\}\rightarrow \{1, \ldots, p\}$ such that $a_{\mu(0)}, \ldots, a_{\mu(n)}$ are linearly independent. Denote the Weil function by $\lambda_{H_{j}}(f(z))=\log\frac{\|f(z)\|\|a_{j}\|}{|<f(z), a_{j}>|}$ for all $j\in\{1,\ldots, p\},$ where $<f(z), a_{j}>$ means the inner product. Then we get
\begin{eqnarray*}
&&\int_{S_{m}(r)}\max_{K}\sum_{k\in K}\lambda_{H_{k}}(f(z))\eta_{m}(z)\\
&\leq&\int_{S_{m}(r)}\max_{\mu\in T}\sum_{j=0}^{n}\log\frac{\|f(z)\|\|a_{\mu(j)}\|}{|<f(z), a_{\mu(j)}>|}\eta_{m}(z)\\
&\leq&\int_{S_{m}(r)}\log\left\{\max_{\mu\in T}\frac{\|f(z)\|^{n+1}}{|C(<f(z), a_{\mu(0)}>, \ldots, <f(z), a_{\mu(n)}>)|}\right\}\eta_{m}(z)\\&&+\int_{S_{m}(r)}\max_{\mu\in T}\left\{\log\frac{|C(<f(z), a_{\mu(0)}>, \ldots, <f(z), a_{\mu(n)}>)|}{\prod_{j=0}^{n}|<f(z), a_{\mu(j)}>|}\right\}\eta_{m}(z)+O(1)\\
&\leq&\int_{S_{m}(r)}\log\left\{\sum_{\mu\in T}\frac{\|f(z)\|^{n+1}}{|C(<f(z), a_{\mu(0)}>, \ldots, <f(z), a_{\mu(n)}>)|}\right\}\eta_{m}(z)\\&&+\int_{S_{m}(r)}\sum_{\mu\in T}\log\left\{\frac{|C(<f(z), a_{\mu(0)}>, \ldots, <f(z), a_{\mu(n)}>)|}{\prod_{j=0}^{n}|<f(z), a_{\mu(j)}>|}\right\}\eta_{m}(z)+O(1)\\
&:=&I_{1}+I_{2}+O(1),
\end{eqnarray*}where $C(<f(z), a_{\mu(0)}>, \ldots, <f(z), a_{\mu(n)}>)$ denotes the $q$-Casorati determinant of $<f(z), a_{\mu(0)}>,$ $\ldots,$ $<f(z), a_{\mu(n)}>.$\par

Set $h_{\mu(j)}:=<f(z), a_{\mu(j)}>$ and $a_{\mu(j)}=(a_{\mu(j),0}, \ldots, a_{\mu(j),n})$ for all $0\leq j\leq n.$ Then we have
\begin{eqnarray*}
&&|C(<f(z), a_{\mu(0)}>, \ldots, <f(z), a_{\mu(n)}>)|\\
&=&\left|\begin{array}{cccc}
         h_{\mu(0)} & h_{\mu(1)} & \cdots & h_{\mu(n)} \\
         \overline{h}_{\mu(0)}^{[1]} & \overline{h}_{\mu(1)}^{[1]} & \cdots & \overline{h}_{\mu(n)}^{[1]} \\
         \vdots & \vdots & \vdots & \vdots \\
         \overline{h}_{\mu(0)}^{[n]} & \overline{h}_{\mu(1)}^{[n]} & \cdots & \overline{h}_{\mu(n)}^{[n]} \\
       \end{array}
\right|\\
&=&\left|\begin{array}{cccc}
         f_{0} & f_{1} & \cdots & f_{n} \\
         \overline{f}_{0}^{[1]} & \overline{f}_{1}^{[1]} & \cdots & \overline{f}_{n}^{[1]} \\
         \vdots & \vdots & \vdots & \vdots \\
         \overline{f}_{0}^{[n]} & \overline{f}_{n} & \cdots & \overline{f}_{n}^{[n]} \\
       \end{array}
\right|\cdot  \left|\begin{array}{cccc}
         a_{\mu(0),0} & a_{\mu(1), 0} & \cdots & a_{\mu(n), 0} \\
         a_{\mu(0),1} & a_{\mu(1), 1} & \cdots & a_{\mu(n), 1} \\
         \vdots & \vdots & \vdots & \vdots \\
  a_{\mu(0), n} & a_{\mu(1), n} & \cdots & a_{\mu(n), n} \\
              \end{array}
\right|\\
&:=&|C(f_{0}, \ldots, f_{n})|\cdot A_{\mu}.
\end{eqnarray*} By Lemma \ref{L-4.1} we know that $C(f_{0}, \ldots, f_{n})\not\equiv 0$ since $f$ is linearly nondegenerate over the field $\mathcal{P}_{q}.$ Hence, by the definition of the characteristic function and the Jensen's Formula, we have \begin{eqnarray*}
I_{1}&=&\int_{S_{m}(r)}\log\left\{\sum_{\mu\in T}\frac{\|f(z)\|^{n+1}}{|C(<f(z), a_{\mu(0)}>, \ldots, <f(z), a_{\mu(n)}>)|}\right\}\eta_{m}(z)\\
&\leq&\int_{S_{m}(r)}\log\|f(z)\|^{n+1}\eta_{m}(z)-\int_{S_{m}(r)}\log |C(f_{0}, \ldots, f_{n})|\eta_{m}(z)+O(1)\\
&=&(n+1)T_{f}(r)-N(r,\frac{1}{C(f_{0}, \ldots, f_{n})})+O(1).
\end{eqnarray*}
Now denote $g_{\mu(l)}=\frac{<f(z), a_{\mu(l)}>}{<f(z), a_{\mu(0)}>}=\frac{h_{\mu(l)}}{h_{\mu(0)}}$ for $0\leq l\leq n.$ Then
\begin{eqnarray*}
&&\frac{|C(h_{\mu(0)}, h_{\mu(1)}, \ldots, h_{\mu(n)})|}{|h_{\mu(0)}\overline{h}_{\mu(1)}\cdots \overline{h}^{[n]}_{\mu(n)}|}= \left|\frac{\left|\begin{array}{cccccc}
h_{\mu(0)} & h_{\mu(1)} &\cdots & h_{\mu(n)} \\
\overline{h}_{\mu(0)} & \overline{h}_{\mu(1)}& \cdots & \overline{h}_{\mu(n)} \\
\vdots&\vdots&\vdots&\vdots\\
\overline{h}_{\mu(0)}^{[n]} & \overline{h}_{\mu(1)}^{[n]} &\cdots & \overline{h}_{\mu(n)}^{[n]}\\
\end{array}\right|}{h_{\mu(0)}\cdot \overline{h}_{\mu(1)}\cdots\overline{h}_{\mu(n)}^{[n]}}\right|\\&=& \left|\frac{\left|\begin{array}{ccccccc}
1 & g_{\mu(1)} & \cdots & g_{\mu(n)}\\
1 & \overline{g}_{\mu(1)}&\cdots & \overline{g}_{\mu(n)}\\
\vdots&\vdots&\vdots&\vdots\\
1 & \overline{g}_{\mu(1)}^{[n]}& \cdots & \overline{g}_{\mu(n)}^{[n]}\\
\end{array}\right|}{\overline{g}_{\mu(1)}\cdot\overline{g}_{\mu(2)}^{[2]}\cdots\overline{g}_{\mu(n)}^{[n]}}\right|
=\left|\frac{\left|\begin{array}{ccccccc}
1 & 1 & \cdots & 1\\
1 & \frac{\overline{g}_{\mu(1)}}{g_{\mu(1)}}&\cdots & \frac{\overline{g}_{\mu(n)}}{g_{\mu(n)}} \\
\vdots&\vdots&\vdots&\vdots\\
1 & \frac{\overline{g}_{\mu(1)}^{[n]}}{g_{\mu(n)}} &\cdots & \frac{\overline{g}_{\mu(n)}^{[n]}}{g_{\mu(n)}}\\
\end{array}\right|}{\frac{\overline{g}_{\mu(1)}}{g_{\mu(1)}}
\cdot\frac{\overline{g}_{\mu(2)}^{[2]}}{g_{\mu(2)}}
\cdots\frac{\overline{g}_{\mu(n)}^{[n]}}{g_{\mu(n)}}}\right|\\
&=&
\frac{\sum_{i_{1}+\ldots+i_{n}\leq\frac{n(n+1)}{2}}
\sum_{l=1}^{n}\frac{|\overline{g}_{\mu(l)}^{[i_{l}]}|}{|g_{\mu(l)}|}}{\frac{|\overline{g}_{\mu(1)}|}{|g_{\mu(1)}|}
\cdot\frac{|\overline{g}_{\mu(2)}^{[2]}|}{|g_{\mu(2)}|}
\cdots\frac{|\overline{g}_{\mu(n)}^{[n]}|}{|g_{\mu(n)}|}}.
\end{eqnarray*}
By the Jensen's Formula, we have
\begin{eqnarray*}&&\int_{S_{m}(r)}\log\frac{|h_{\mu(0)}\overline{h}_{\mu(1)}\cdots \overline{h}^{[n]}_{\mu(n)}|}{|h_{\mu(0)}h_{\mu(1)}\cdots h_{\mu(n)}|}\eta_{m}(z)\\
&\leq&\sum_{j=0}^{n}\int_{S_{m}(r)}\log|\overline{h}_{\mu(j)}^{[j]}|\eta_{m}(z)
-\sum_{j=0}^{n}\int_{S_{m}(r)}\log |h_{\mu(j)}|\eta_{m}(z)+O(1)\\
&\leq&\sum_{j=0}^{n}N\bigg(r, \frac{1}{\overline{h}_{\mu(j)}^{[j]}}\bigg)
-\sum_{j=0}^{n}N\left(r, \frac{1}{h_{\mu(j)}}\right)+O(1).
\end{eqnarray*}
Hence, we have
\begin{eqnarray*} I_{2}&=&\int_{S_{m}(r)}\sum_{\mu\in T}\log\left\{\frac{|C(<f(z), a_{\mu(0)}>, \ldots, <f(z), a_{\mu(n)}>)|}{\prod_{j=0}^{n}|<f(z), a_{\mu(j)}>|}\right\}\eta_{m}(z)\\
&\leq&\sum_{\mu\in T}\int_{S_{m}(r)}\log\left\{\frac{|C(h_{\mu(0)}, h_{\mu(1)}, \ldots, h_{\mu(n)})|}{|h_{\mu(0)}\overline{h}_{\mu(1)}\cdots \overline{h}^{[n]}_{\mu(n)}|}\right\}\eta_{m}(z)\\&&+\sum_{\mu\in T}\int_{S_{m}(r)}\log\frac{|h_{\mu(0)}\overline{h}_{\mu(1)}\cdots \overline{h}^{[n]}_{\mu(n)}|}{|h_{\mu(0)}h_{\mu(1)}\cdots h_{\mu(n)}|}\eta_{m}(z)\\
&\leq& \sum_{\mu\in T}\sum_{i_{1}+\ldots+i_{n}\leq\frac{n(n+1)}{2}}
\sum_{l=1}^{n}\int_{S_{m}(r)}\log^{+}\frac{|\overline{g}_{\mu(l)}^{[i_{l}]}|}{|g_{\mu(l)}|}\eta_{m}(z)+O(1)\\
&&+\sum_{\mu\in T}\sum_{l=1}^{n}\int_{S_{m}(r)}\log^{+}
\frac{|g_{\mu(l)}|}{|\overline{g}_{\mu(l)}^{[l]}|}\eta_{m}(z)+\sum_{\mu\in T}\sum_{j=0}^{n}\left(N\bigg(r, \frac{1}{\overline{h}_{\mu(j)}^{[j]}}\bigg)-N\bigg(r, \frac{1}{h_{\mu(j)}}\bigg)\right).
\end{eqnarray*}
Therefore, it follows that
\begin{eqnarray*}
&&\int_{S_{m}(r)}\max_{K}\sum_{k\in K}\lambda_{H_{k}}(f(z))\eta_{m}(z)\\
& \leq &\int_{S_{m}(r)}\max_{\mu\in T}\sum_{j=0}^{n}\log\frac{\|f(z)\|\|a_{\mu(j)}\|}{|<f(z), a_{\mu(j)}>|}\eta_{m}(z)\\
&\leq &(n+1)T_{f}(r)-N(r,\frac{1}{C(f_{0}, \ldots, f_{n})})+O(1)\\
&&+\sum_{\mu\in T}\sum_{i_{1}+\ldots+i_{n}\leq\frac{n(n+1)}{2}}
\sum_{l=1}^{n}\int_{S_{m}(r)}\log^{+}\frac{|\overline{g}_{\mu(l)}^{[i_{l}]}|}{|g_{\mu(l)}|}\eta_{m}(z)\\
&&+\sum_{\mu\in T}\sum_{l=1}^{n}\int_{S_{m}(r)}\log^{+}
\frac{|g_{\mu(l)}|}{|\overline{g}_{\mu(l)}^{[l]}|}\eta_{m}(z)+\sum_{\mu\in T}\sum_{j=0}^{n}\left(N\bigg(r, \frac{1}{\overline{h}_{\mu(j)}^{[j]}}\bigg)-N\bigg(r, \frac{1}{h_{\mu(j)}}\bigg)\right).
\end{eqnarray*}Hence, the proof is completed.
 \end{proof}
 \endgroup

For the special case when $q=(\tilde{q}, \ldots, \tilde{q})\in\mathbb{C}^{m}$  and $\zeta(f)=0,$ we get the following second main theorem.\par

\begin{corollary}\label{HSMT} Let $q=(\tilde{q}, \ldots, \tilde{q})\in\mathbb{C}^{m}\setminus\{0, 1\},$ $c=(c_{1}, \ldots, c_{m})\in\mathbb{C}^{m},$ and let $f=[f_{0}, \ldots, f_{n}]:\mathbb{C}^{m}\rightarrow\mathbb{P}^{n}(\mathbb{C})$ be a meromorphic mapping with zero order and linearly nondegenerate over the field $\mathcal{P}_{q}^{0}.$ Let $H_{1}, \ldots, H_{p}$ (defining coefficient vectors $a_{1}, \ldots, a_{p}$ respectively) be arbitrary hyperplanes  in $\mathbb{P}^{n}(\mathbb{C}).$  Let $T$ be the set of all injective maps $\mu: \{0, 1, \ldots, n\}\rightarrow \{1, \ldots, p\}$ such that $a_{\mu(0)}, \ldots, a_{\mu(n)}$ are linearly independent. Then we have
\begin{eqnarray*}
&&\int_{S_{m}(r)}\max_{K}\sum_{k\in K}\lambda_{H_{k}}(f(z))\eta_{m}(z)\\
&\leq&\int_{S_{m}(r)}\max_{\mu\in T}\sum_{j=0}^{n}\log\frac{\|f(z)\|\|a_{\mu(j)}\|}{|<f(z), a_{\mu(j)}>|}\eta_{m}(z)\\
&\leq&(n+1)T_{f}(r)-N\bigg(r,\frac{1}{C(f_{0}, \ldots, f_{n})}\bigg)+o(T_{f}(r))
\end{eqnarray*} for all $r=\|z\|$ on a set of logarithmic density one, where the maximum is taken over all subsets $K$ of $\{1, \ldots, p\}$ such that $a_{j}$ $(j\in K)$ are linearly independent.
\end{corollary}

\begin{proof} Denote by $g_{\mu(l)}=\frac{<f(z), a_{\mu(l)}>}{<f(z), a_{\mu(0)}>}$ for $0\leq l\leq n.$  Then by the definition of the characteristic function, we have
$$T_{g_{\mu(l)}}(r)\leq T_{f}(r).$$
Note that the order of $f$ is zero, and thus the order of the meromorphic functions $g_{\mu(l)}$ $(0\leq l\leq n)$ on $\mathbb{C}^{m}$ are all zero. Whenever $q=(\tilde{q}, \ldots, \tilde{q})\in\mathbb{C}^{m}\setminus\{0, 1\},$  by Theorem \ref{q-LDL} we get that
\begin{eqnarray*} \int_{S_{m}(r)}
\log^{+}\left|\frac{\overline{g}_{\mu(l)}^{[i_{l}]}}{g_{\mu(l)}}\right|\eta_{m}(z)+\int_{S_{m}(r)}
\log^{+}\left|\frac{g_{\mu(l)}}{\overline{g}_{\mu(l)}^{[l]}}\right|\eta_{m}(z)
=o(\max_{0\leq l\leq n}\{T_{g_{\mu(l)}}(r)\})=o(T_{f}(r))\end{eqnarray*}
for all $r=\|z\|$ on a set of logarithmic density one.\par

Furthermore, by the Jensen's Formula and the definition of characteristic function, we get that for any $\mu\in T$ and  $j\in \{0, 1, \ldots, n\},$
\begin{eqnarray*}
  N\left(r, \frac{1}{h_{\mu(j)}}\right)&=&\int_{S_{m}(r)}\log|h_{\mu(j)}|\eta_{m}(z)+O(1)\\
  &\leq&\int_{S_{m}(r)}\log\max\{|f_{0}(z)|, \ldots, |f_{n}(z)|\}\eta_{m}(z)+O(1)\\
 &=&T_{f}(r)+o(T_{f}(r)),
\end{eqnarray*} and thus
 \begin{eqnarray*}\lambda_{j}:=\limsup_{r\rightarrow\infty} \frac{\log^{+} N(r, \frac{1}{h_{\mu(j)}})}{\log r}\leq \zeta(f)=0.\end{eqnarray*} Hence by Lemma \ref{L-8.2} (see below in Section~\ref{TCsec}), we get that for  all $\mu \in T$ and $j\in\{0, 1, \ldots, n\},$
\begin{eqnarray*}
N\bigg(r, \frac{1}{\overline{h}_{\mu(j)}^{[j]}}\bigg)&\leq& (1+o(1))N\left(r, \frac{1}{h_{\mu(j)}}\right)\\
 &\leq& N\left(r, \frac{1}{h_{\mu(j)}}\right)+o(T_{f}(r)).
 \end{eqnarray*}
Therefore, the theorem immediately follows from Theorem \ref{GSMT}. \end{proof}

By a careful analysis of the case where the hyperplanes are in general position and the map $f$ is linearly nondegenerate,  we have the $q$-difference analogue of the Cartan's second main theorem with hyperplanes from Corollary~\ref{HSMT}.\par

\begin{theorem}\label{T-4.2} Let $q=(\tilde{q}, \ldots, \tilde{q})\in\mathbb{C}^{m}\setminus\{0, 1\},$  $c\in\mathbb{C}^{m},$ and let $f=[f_{0}, \ldots, f_{n}]:\mathbb{C}^{m}\rightarrow\mathbb{P}^{n}(\mathbb{C})$ be a meromorphic map with zero order and linearly nondegenerate over the field $\mathcal{P}_{q}^{0}.$ Let $H_{j}$ $(1\leq j\leq p)$ (defining polynomials $P_{1}, \ldots, P_{p}$ of degree one, respectively) be hyperplanes located in general position in $\mathbb{P}^{n}(\mathbb{C}).$  Then
\begin{eqnarray*}&&(p-n-1)T_{f}(r)\leq \sum_{j=1}^{p}N\left(r, \frac{1}{P_{j}\circ f(z)}\right)-N\left(r,\frac{1}{C(f_{0}, \ldots, f_{n})}\right)+o(T_{f}(r))
\end{eqnarray*} for all $r=\|z\|$ on a set of logarithmic density one, where $C(f_{0}, \ldots, f_{n})$ is the $q$-Casorati determinant of $f=(f_{0}, \ldots, f_{n}).$
\end{theorem}

\begin{proof}
Let $a_{j}$ be the coefficient vectors of the hyperplane $H_{j},$ $1\leq j\leq p.$ By the assumption that $H_{1}, \ldots, H_{p}$ are located in general position in $\mathbb{P}^{n}(\mathbb{C}),$ we can solve the linear equations system
\begin{eqnarray*}
                   <f, a_{\mu(0)}> &=& a_{\mu(0), 0}f_{0}+\ldots+ a_{\mu(0), n}f_{n}, \\
                  <f, a_{\mu(1)}> &=& a_{\mu(1), 0}f_{0}+\ldots+ a_{\mu(1), n}f_{n}, \\
                    \vdots  \\
                  <f, a_{\mu(n)}> &=& a_{\mu(n), 0}f_{0}+\ldots+ a_{\mu(n), n}f_{n}, \\
                 \end{eqnarray*}
and get that
\begin{eqnarray*}
f_{j}=\tilde{a}_{\mu(j), 0} <f, a_{\mu(0)}>+\ldots+\tilde{a}_{\mu(j), n} <f, a_{\mu(n)}>, \quad 0\leq j\leq n,
\end{eqnarray*} where $(\tilde{a}_{\mu(i), j})_{(n+1)\times(n+1)}$ is the inverse matrix of $(a_{\mu(i), j})_{(n+1)\times(n+1)}.$ Hence for any $\mu\in T\subset\{1, \ldots, p\},$ $\sharp T=n+1,$ there exists a positive number $A$ such that
$$\|f(z)\|\leq A\max_{0\leq j\leq n}\left\{\frac{|<f(z), a_{\mu(j)}>|}{\|a_{\mu(j)}\|}\right\}.$$
It is clear that for any given $z\in\mathbb{C}^{m},$ there always exists a $\mu\in  T\subset\{1, \ldots, p\},$ $\sharp T=n+1,$ such that
$$0\leq\frac{|<f(z), a_{\mu(0)}>|}{\|a_{\mu(0)}\|}\leq \ldots\leq \frac{|<f(z), a_{\mu(n)}>|}{\|a_{\mu(n)}\|}\leq \frac{|<f(z), a_{j}>|}{\|a_{j}\|},$$ for $j\neq \mu(i),$ $i=0,1,\ldots,n.$ Hence we obtain
$$\prod_{j=1}^{p}\frac{\|f(z)\|\|a_{j}\|}{|<f(z), a_{j}>|}\leq A^{p-n-1}\max_{\mu\in T}\prod_{i=0}^{n}\left\{\frac{\|f(z)\|\|a_{\mu(i)}\|}{|<f(z), a_{\mu(i)}>|}\right\}.$$
Combining this with Corollary \ref{HSMT} gives \begin{eqnarray*}
\sum_{j=1}^{p}m_{f}(r, H_{j})&=&\int_{S_{m}(r)}\log\prod_{j=1}^{p}\frac{\|f(z)\|\|a_{j}\|}{|<f(z), a_{j}>|}\eta_{m}(z)\\
&\leq&\int_{S_{m}(r)}\max_{\mu\in T}\sum_{j=0}^{n}\log\frac{\|f(z)\|\|a_{\mu(j)}\|}{|<f(z), a_{\mu(j)}>|}\eta_{m}(z)+O(1)\\
&\leq&(n+1)T_{f}(r)-N\bigg(r,\frac{1}{C(f_{0}, \ldots, f_{n})}\bigg)+o(T_{f}(r))
\end{eqnarray*} for all $r=\|z\|$ on a set of logarithmic density one. And then by the first main theorem, the theorem is immediately obtained.
\end{proof}

\begin{remark}Set
$ L=\frac{\prod_{j=1}^p H_j(f)}{C(f)}.$ Clearly, Both $\prod_{j=1}^p H_j(f)$ and $C(f)$ are entire functions on $\mathbb{C}^{m}$. By the Jensen's Formula, we have
\begin{eqnarray*}&&N\left(r,\frac{1}{L}\right)-N(r,L)\\
&=&\int_{S_{m}(r)}\log|L(z)|\eta_{m}(z)+O(1)\\
&=&\sum_{j=1}^{p}\int_{S_{m}(r)}\log|H_j(f)(z)|\eta_{m}(z)-\int_{S_{m}(r)}\log|C(f)(z)|\eta_{m}(z)+O(1)\\
&=&\sum_{j=1}^{p} N\left(r,\frac{1}{H_j(f)}\right)-N\left(r,\frac{1}{C(f)}\right)+O(1).
\end{eqnarray*}Hence the conclusion of Theorem \ref{HSMT} can be written as
\begin{eqnarray*}
(p-(n+1))T_f(r)
\leq N\bigg(r,\frac{1}{L}\bigg)-N(r,L)+o(T_f(r)),
\end{eqnarray*} which is a $q$-difference counterpart of the Gundersen-Hayman version \cite{gundersen-hayman} of the Cartan's the second main theorem with hyperplanes in several complex variables, and may also be applied into complex functional equations.
\end{remark}

By adopting the method of Ru \cite{ru-2004} (see also \cite{corvaja-zannier, an-phuong}) and using Theorem~\ref{q-LDL}  and Theorem \ref{HSMT}, we also get the second main theorem for hypersurfaces in general position, in which if the hypersurfaces reduce to hyperplanes and the map reduces to a linearly nondegenerated one, then by taking $\alpha=d=d_{j}=1$ the result, stated as follows, implies Theorem \ref{T-4.2}. We omit the details of the proof. \par

\begin{theorem}\label{SMT} Let $q=(\tilde{q}, \ldots, \tilde{q})\in\mathbb{C}^{m}\setminus\{0, 1\},$ $c\in\mathbb{C}^{m},$ and let $f=[f_{0}, \ldots, f_{n}]:\mathbb{C}^{m}\rightarrow\mathbb{P}^{n}(\mathbb{C})$ be a meromorphic map with zero order and algebraically nondegenerate over the field $\mathcal{P}_{q}^{0}.$ Let $Q_{j}$ $(1\leq j\leq p)$ (defining homogeneous polynomials $D_{j}$ of degree $d_{j}$) be hypersurfaces of degree $d_{j} (1\leq j\leq p)$ located in general position in $\mathbb{P}^{n}(\mathbb{C}).$ Let $d$ be the least common multiple of the $d_{j}.$ Then there exists a large positive integer $\alpha$ which is divisible by $d,$ such that
\begin{eqnarray*}&&(p-n-1)T_{f}(r)\\&\leq& \sum_{j=1}^{p}\frac{1}{d_{j}}N\left(r, \frac{1}{D_{j}\circ f(z)}\right)-\frac{1}{\frac{\alpha^{n+1}}{(n+1)!}+O(\alpha^{n})}N\left(r,\frac{1}{C(f^{I_{1}}, \ldots, f^{I_{M}})}\right)+o(T_{f}(r))
\end{eqnarray*} for all $r=\|z\|$ on a set of logarithmic density one, where $I_{j}=(i_{j0}, \ldots, i_{jn})$, $\sharp I_{j}=i_{j0}+\ldots+i_{jn}=\alpha$, and $M=\left(\begin{array}{c}
                                                          \alpha+n \\
                                                          n
                                                        \end{array}
\right).$
\end{theorem}

\section{Difference analogues of generalized Picard-type theorems}\label{picard_sec}

Fujimoto \cite{fujimoto-5} and Green \cite{green-1} gave a natural generalization of the Picard's theorem by showing that if $f: \mathbb{C}\rightarrow\mathbb{P}^{n}(\mathbb{C})$ omits $n+p$ hyperplanes in general position where $p\in\{1, \ldots, n+1\},$ then the image of $f$ is contained in a linear subspace of dimension at most $[\frac{n}{p}].$ In 2014, Halburd, the second author and Tohge \cite[Theorem 6.1]{halburd-korhonen-tohge-1} proposed a $q$-difference analogue of the general Picard-type theorem for holomorphic curves with zero order.\par

\begin{theorem}[\cite{halburd-korhonen-tohge-1}] \label{T-7.1} Let $f: \mathbb{C}\rightarrow\mathbb{P}^{n}(\mathbb{C})$ be a holomorphic curve with zero order. Let $q\in \mathbb{C}\setminus\{0, 1\},$ and let $p\in \{1, \ldots, n+1\}.$ If $p+n$ hyperplanes in general position in $\mathbb{P}^{n}(\mathbb{C})$ have forward invariant preimages under $f$ with respect to the rescaling $\tau(z)=qz,$ then the image of $f$ is contained in a projective linear subspace over $\mathcal{P}_{q}^{0}$ of dimension $\leq [\frac{n}{p}].$
\end{theorem}

Here we say that the pre-image of $H(f(z))$ for a meromorphic mapping $f: \mathbb{C}^{m}\rightarrow \mathbb{P}^{n}(\mathbb{C})$ intersecting a hyperplane $H$ of $\mathbb{P}^{n}(\mathbb{C})$ is forward invariant with respect to the mapping $\tau=q\cdot z+c$ if $\tau(H(f)^{-1})\subset H(f)^{-1}$  where  $\tau (H(f)^{-1})$ and $H(f)^{-1}$  are considered to be multi-sets in which each point is repeated according to its multiplicity. By this definition the (empty and thus forward invariant) pre-images of the usual Picard exceptional values become special cases of forward invariant pre-images.

In this section, by the $q$-difference version of second main theorem given in the section above, we extend Theorem \ref{T-7.1} to the case of meromorphic mappings $f:\mathbb{C}^{m}\rightarrow\mathbb{P}^{n}(\mathbb{C})$ intersecting hyperplanes in general position which have forward invariant preimages under $f$ with respect to the linear mapping $\tau(z)=qz+c,$ where $q\in \mathbb{C}\setminus\{0, 1\}$ and $c=(c_{1}, \ldots, c_{m}).$\par

\begin{theorem}\label{T-7.2} Let $q\in\mathbb{C}\setminus\{0, 1\},$ $c\in\mathbb{C}^{m},$ let $p\in\{1, \ldots, n+1\}.$ Assume that $f$ is a meromorphic mapping from $\mathbb{C}^{m}$ into $\mathbb{P}^{n}(\mathbb{C})$ of zero order. If $p+n$ hyperplanes in general position in $\mathbb{P}^{n}(C)$ have forward invariant preimages under $f$ with respect to the linear mapping
$\tau(z)=qz+c,$ then the image of $f$ is contained in a projective linear subspace over $\mathcal{P}_{q}^{0}$ of dimension $\leq [\frac{n}{p}].$
\end{theorem}

Before proving Theorem \ref{T-7.2}, we need two lemmas as follows.\par

\begin{lemma}\label{L-7.1}
Let $q\in \mathbb{C}\setminus\{0, 1\},$ $c\in\mathbb{C}^{m},$ and $f=[f_{0}, \ldots, f_{n}]$ be a meromorphic mapping from $\mathbb{C}^{m}$ into $\mathbb{P}^{n}(\mathbb{C})$ of zero order, and let all zeros of $f_{0}, \ldots, f_{n}$ be forward invariant with respect to the linear mapping $\tau(z)=qz+c.$ If $\frac{f_{i}}{f_{j}}\not\in \mathcal{P}^{0}_{q}$ for all $i, j\in \{0,\ldots,n\}$ such that $i\neq j,$ then $f$ is linearly nondegenerate over $\mathcal{P}^{0}_{q}.$
\end{lemma}

\begin{proof}
Assume that the conclusion is not true, that is there exist $A_{0}, \ldots, A_{n}\in \mathcal{P}^{0}_{q}$ such that
\begin{equation}\label{E-6.1} A_{0}f_{0}+\cdots+A_{n-1}f_{n-1}=A_{n}f_{n}\end{equation}
and such that not all $A_{j}$ are identically zero. Without loss of generality we may assume that none of $A_{j}$ are identically zero. Since all zeros of $f_{0}, \ldots, f_{n}$ are forward invariant with respect to the linear mapping $\tau(z)=qz+c$ and since $A_{0}, \ldots, A_{n}\in \mathcal{P}^{0}_{q},$ we can choose a meromorphic function $F$ on $\mathbb{C}^{m}$ such that $FA_{0}f_{0}, \ldots, FA_{n}f_{n}$ are holomorphic functions on $\mathbb{C}^{m}$ without common zeros and such that the preimages of all  zeros of $FA_{0}f_{0}, \ldots, FA_{n}f_{n}$ are forward invariant with respect to the linear mapping $\tau(z)=qz+c.$ Then we have \begin{eqnarray}\label{E-6.2}\limsup_{r\rightarrow\infty}\frac{\log^{+}\left(N(r, \frac{1}{F})+N(r, F)\right)}{\log r}=0\end{eqnarray}
and $FA_{0}f_{0},\ldots, FA_{n-1}f_{n-1}$ cannot have any common zeros.\par

Denote $g_{j}:=FA_{j}f_{j}$ for $0\leq j\leq n.$ Then $T_{G}(r)$ is well defined for $G=[g_{0}, \ldots, g_{n-1}]$, which is a holomorphic mapping from $\mathbb{C}^{m}$ into $\mathbb{P}^{n-1}(\mathbb{C}).$ Then by the definition of Nevanlinna-Cartan's characteristic function and the Jensen's Formula, we have
\begin{eqnarray*}T_{G}(r)&=&\int_{S_{m}(r)}\log\|G\|\eta_{m}(z)+O(1)\\
&\leq&\int_{S_{m}(r)}\log |F(z)|\eta_{m}(z)+\int_{S_{m}(r)}\log\|f(z)\|\eta_{m}(z)\\
&&+\sum_{j=0}^{n-1}\int_{S_{m}(r)}\log^{+}|A_{j}|\eta_{m}(z)+O(1)\\
&\leq& N\left(r, \frac{1}{F}\right)-N(r, F)+T_{f}(r)+\sum_{j=0}^{n-1}T_{A_{j}}(r)
\end{eqnarray*}
which together with \eqref{E-6.2} imply that the order of $G$ satisfies $\zeta(G)=0.$\par

Assume that the meromorphic mapping $G: \mathbb{C}^{m}\rightarrow\mathbb{P}^{n-1}(\mathbb{C})$ is linearly nondegenerate over $\mathcal{P}_{q}^{0}.$ Then by Lemma \ref{L-4.1}, it follows that  $C(g_{0}, \ldots, g_{n-1})\not\equiv 0.$ Define the following hyperplanes
$$H_{j}:\quad  w_{j}=0, \quad 0\leq j\leq n-1,$$
and $$H_{n}:\quad w_{0}+w_{1}+\ldots+w_{n-1}=0,$$ where $[w_{0}, \ldots, w_{n-1}]$ is a homogeneous coordinate system of $\mathbb{P}^{n-1}(\mathbb{C}).$ So, we have $H_{j}(G(z))=g_{j}(z)$ for $0\leq j\leq n-1$ and $$H_{n}(G(z))=g_{0}(z)+\ldots+g_{n-1}(z)=F(z)A_{n}(z)f_{n}(z)=g_{n}(z).$$ Clearly, the $n+1$ hyperplanes $H_{0}, \ldots, H_{n}$ are located in general position of $\mathbb{P}^{n-1}(\mathbb{C}).$ Hence by Theorem \ref{T-4.2} we have
\begin{eqnarray*}T_{G}(r)&=&
\left((n+1)-(n-1)-1\right)T_{G}(r)\\
&\leq& \sum_{j=0}^{n}N\left(r, \frac{1}{g_{j}}\right)-N\left(r, \frac{1}{C(g_{0}, \ldots, g_{n-1})}\right)+o(T_{G}(r))
\end{eqnarray*} for all $r$ on a set of logarithmic density one.\par

Since the preimages of all zeros of $g_{0},$ $\ldots,$ $g_{n}$ are forward invariant with
respect to $\tau(z)=qz+c,$ all zeros of $g_{j}, j = 0,\ldots , n-1,$ are zeros of the $q$-Casorati determinant $C(g_{0}, \ldots , g_{n-1})$ with the same or higher multiplicity. Moreover, since $g_{0},\ldots, g_{n}$ do not have any common zeros, it follows in particular that for each $z_{0} \in\mathbb{C}^{m}$ such that $g_{n}(z_{0})= 0$ with multiplicity $m_{0}$ there exists $k_0 \in\{0,\ldots , n-1\}$ such that $g_{k_{0}}:=FA_{k_{0}}f_{k_{0}}\not\equiv 0.$  Using \eqref{E-6.1} we may write
\begin{eqnarray*}C(g_{0}, \ldots , g_{n-1})=
C(g_{0},\ldots , g_{k_{0}-1}, g_{n}, g_{k_{0}+1},\ldots, g_{n-1})\end{eqnarray*}
which implies that $C(g_{0}, \ldots , g_{n-1})$ has a zero at $z_0$ with multiplicity
$m_0$ at least. Also, at any common zero the functions $g_{j_k}:=FA_{j_k}
f_{j_k}$ with multiplicities $m_{j_k},$ $k=1, \ldots, l,$ where $\{j_1, \ldots , j_l\} \subset \{1, \ldots, n\}$ and $l \leq n-2,$ the Casorati
determinant $C(g_{0}, \ldots, g_{n-1})$ has a zero of multiplicity $\geq\sum_{k=1}^{l}m_{j_k}.$ This implies
\begin{eqnarray*}
\sum_{j=0}^{n}N\left(r, \frac{1}{g_{j}}\right)\leq N\left(r, \frac{1}{C(g_{0}, \ldots, g_{n-1})}\right).
\end{eqnarray*}
Hence, it follows that $T_{G}(r)=o(T_{G}(r))$ for all $r$ on a set of logarithmic density one, which is an contradiction.\par

Therefore, the meromorphic mapping $G: \mathbb{C}^{m}\rightarrow\mathbb{P}^{n-1}(\mathbb{C})$ is linearly degenerate over $\mathcal{P}_{q}^{0},$ and thus there exist $B_{0}, \ldots, B_{n-1}\in \mathcal{P}^{0}_{q}$ such that
$$B_{0}f_{0}+\cdots+B_{n-2}f_{n-2}=B_{n-1}f_{n-1}$$
and such that not all $B_{j}$ are identically zero. By repeating similar discussions as above it follows that there exist $L_{i}, L_{j}\in\mathcal{P}_{q}^{0}$ such that $$L_{i}f_{i}=L_{j}f_{j}$$ for some $i\neq j$ and not all $L_{i}$ and $L_{j}$ are identically zero. This contradicts
the condition that $\frac{f_{i}}{f_{j}}\not\in \mathcal{P}^{0}_{q}$ for all $\{i,j\}\subset\{0,\ldots,n\}.$  Therefore, the proof is completed.
\end{proof}

The following lemma is  a $q$-difference analogue of the Borel's theorem.\par

\begin{lemma}\label{L-7.2}
Let $q\in \mathbb{C}\setminus\{0, 1\},$ $c\in\mathbb{C}^{m},$ and $f=[f_{0}, \ldots, f_{n}]$ be a meromorphic mapping from $\mathbb{C}^{m}$ into $\mathbb{P}^{n}(\mathbb{C})$ of zero order, and let all zeros of $f_{0}, \ldots, f_{n}$ be forward invariant with respect to the linear mapping $\tau(z)=qz+c.$ Let
$$S_{1}\cup\cdots\cup S_{l}$$ be the partition of $\{0, 1, \ldots, n\}$ formed in such a way that $i$ and $j$ are
in the same class $S_{k}$ if and only if $\frac{f_{i}}{f_{j}}\in \mathcal{P}^{0}_{q}.$ If
$$f_{0}+\ldots+f_{n}=0,$$
then $$\sum_{j\in S_{k}}f_{j}=0$$ for all $k\in\{1,\ldots,l\}.$
\end{lemma}

\begin{proof} Suppose that $i\in S_{k},$ $k\in\{0,\ldots, l\}.$ Then by the condition of the lemma, $f_{i}=A_{i, j_{k}}f_{j_{k}}$ for some $A_{i, j_{k}}\in\mathcal{P}_{q}^{0}$ whenever the indexes $i$ and $j_{k}$ are in the same class $S_{k}.$ This implies that
$$0=\sum_{k=0}^{n}f_{k}=\sum_{k=1}^{l}\sum_{i\in S_{k}}A_{i,j_{k}}f_{j_{k}}=\sum_{k=1}^{l}B_{k}f_{j_{k}}$$ where $B_{k}=\sum_{i\in S_{k}}A_{i,j_{k}}\in\mathcal{P}_{q}^{0}.$ This gives that $f_{j_{1}}, \ldots, f_{j_{l}}$ are linearly degenerate over $\mathcal{P}_{q}^{0}.$ Hence by Lemma \ref{L-7.1} all $B_{k}$ $(k=1,\ldots, l)$ are identically zero. Thus it follows
$$\sum_{i\in S_{k}}f_{i}=\sum_{i\in S_{k}}A_{i,j_{k}}f_{j_{k}}=B_{k}f_{j_{k}}\equiv 0$$
for all $k=\{1,\ldots ,l\}.$
\end{proof}

\begin{proof}[Proof of Theorem \ref{T-7.2}.]
We modify the method of proof of \cite[Theorem~1]{green-1} as follows. Denote $T=\{1, \ldots, n+p\}.$ Let $H_{j}$ be defined by
$$H_{j}:\quad h_{j0}(z)w_{0}+\ldots+h_{jn}(z)w_{n}=0\quad (j\in T)$$ where $[w_{0}, \ldots, w_{n}]$ is a homogeneous coordinate system of $\mathbb{P}^{n}(\mathbb{C}).$ Since $\{H_{j}\}_{j\in T}$ are in general position of $\mathbb{P}^{n}(\mathbb{C}),$ any $n+2$ of $H_{j}$ satisfy a linear relation with nonzero coefficients in $\mathbb{C}.$ By conditions of the theorem, holomorphic functions
$$g_{j}:=H_{j}(f)=h_{j0}f_{0}+\ldots+h_{jn}f_{n}$$ satisfy
$$\{\tau(g_{j}^{-1}(\{0\}))\}\subset \{g_{j}^{-1}(\{0\})\}$$ for all $j\in T,$ where $\{\cdot\}$ denotes a multiset with counting multiplicities of its elements. We say that $i\sim j$ if $g_{i}=\beta g_{j}$ for some $\beta\in\mathcal{P}^{0}_{q}\setminus\{0\}.$  Hence
$$T=\bigcup_{j=1}^{l}S_{j}$$ for some $l\in T.$\par

Firstly, assume that the complement of $S_{k}$ has at least $n+1$ elements for some $k\in\{1, \ldots l\}.$ Choose an element $s_{0}\in S_{k},$ and denote $U=(T\setminus S_{k})\cup\{s_{0}\}.$ Then $U$ contains at least $n+2$ elements, and thus there is a subset $U_{0}\subset U$ such that $U_{0}\cap S_{k}=\{s_{0}\}$ and $\sharp U_{0}=n+2.$ Therefore there exists $\beta_{j}\in\mathbb{C}\setminus\{0\}$ such that
$$\sum_{j\in U_{0}}\beta_{j}H_{j}=0.$$ Hence,
 $$\sum_{j\in U_{0}}\beta_{j}g_{j}=\sum_{j\in U_{0}}\beta_{j}H_{j}(f)=0.$$

Without loss of generality, we may assume that $U_{0}=\{s_{1}, \ldots, s_{n+1}\}\cup\{s_{0}\}.$ It is easy to see from above discussion that all of zeros of $\beta_{j}g_{j}$ $(j\in U_{0})$ are forward invariant with respect to the linear mapping $\tau(z)=qz+c,$ and $$G:=[\beta_{s_{0}}g_{s_{0}}: \beta_{s_{1}}g_{s_{1}}: \cdots: \beta_{s_{n+1}}g_{s_{n+1}}]$$ is a meromorphic mapping from $\mathbb{C}^{m}$ into $\mathbb{P}^{n+1}(\mathbb{C})$ with zero order. Furthermore, $\frac{\beta_{i}g_{i}}{\beta_{s_{0}}g_{s_{0}}}\not\in\mathcal{P}_{q}^{0}$ for any $i\in U_{0}\setminus\{s_{0}\},$ thus $i\not\sim s_{0}.$ Hence by Lemma \ref{L-7.2}  we have $\beta_{s_{0}}g_{s_{0}}=0,$ and thus $H_{s_{0}}(f(z))\equiv 0.$ This means that the image $f(\mathbb{C}^{m})$ is included in the hyperplane $H_{s_{0}}$ of $\mathbb{P}^{n}(\mathbb{C}).$\par

Secondly, assume that the set $T\setminus S_{k}$ has at most $n$ elements. Then $S_{k}$ has at least $p$ elements for all $k=1, \ldots, l.$ This implies that $$l\leq \frac{n+p}{p}.$$\par

Let $V$ be any subset of $T$ with $\sharp V=n+1.$ Then $\{H_{j}\}_{j\in V}$ are linearly independent. Denote $V_{k}:=V\cap S_{k}.$ Then we have
$$V=\bigcup_{k=1}^{l}V_{k}.$$  Since each set $V_{k}$ gives raise to $\sharp (V_{k}-1)$ equations over the field $\mathcal{P}_{q}^{0},$ it follows that there are  at least
\begin{eqnarray*}\sum_{k=1}^{l}(\sharp V_{k}-1)&=&n+1-l\geq n+1-\frac{n+p}{p}\\&=&n-\left(\frac{n}{p}\right)\end{eqnarray*}
 linear independent relations over the field $\mathcal{P}_{q}^{0}.$ This means that the image of $f$ is contained in a linear subspace over $\mathcal{P}_{q}^{0}$ of dimension $\leq [\frac{n}{p}].$ The proof of the theorem is finished.

\end{proof}

According  to the definition of forward invariant pre-image, the following result is an extension of the Picard's theorem under the ``zero order'' growth condition.\par

\begin{theorem}\label{T-7.3} Let $f$ be a meromorphic mapping  with zero order from $\mathbb{C}^{m}$ into $\mathbb{P}^{n}(C),$ and let $\tau(z)=qz+c,$ where $q\in \mathbb{C}\setminus\{0, 1\}$ and $c\in\mathbb{C}^{m}.$ If $\tau((f, H_{j})^{-1})\subset (f,  H_{j})^{-1}$ (counting multiplicity) hold for $p$ distinct hyperplanes $\{H_{j}\}_{j=1}^{p}$ in general position in $\mathbb{P}^{n}(\mathbb{C}),$ and if $p>n+1,$ then $f(z)\equiv f(qz+c).$
\end{theorem}

\begin{proof} By Theorem \ref{T-7.2},  the image of $f$ is contained in a projective linear subspace over $\mathcal{P}_{q}^{0}$ of dimension $\leq [\frac{n}{p}].$
By the assumption $p>n$ it follows  $[\frac{n}{p}]=0.$ Hence $f(z)=f(qz+c).$ The proof of Theorem \ref{T-7.3} is thus completed.\end{proof}

The following corollary follows immediately from the above theorem for $n=1,$  and it can also be seen as a uniqueness theorem for a meromorphic function and  its $q$-shift. \par

\begin{corollary}\label{Cor} Let $f$ be a meromorphic function  with zero order on $\mathbb{C}^{m}$ and let $\tau(z)=qz+c,$ where $q\in \mathbb{C}\setminus\{0, 1\}$ and $c\in\mathbb{C}^{m}.$ If $\tau((f(a_{j})^{-1})\subset (f(a_{j}))^{-1}$ (counting multiplicity) hold for $p$ distinct numbers $\{a_{j}\}_{j=1}^{p}$ in $\mathbb{P},$ and if $p>2,$ then $f(z)\equiv f(qz+c).$
\end{corollary}

\begin{remark}
Let $c=(0, \ldots, 0).$ If the assumptions of Theorem~\ref{T-7.3} are valid, and, in addition, either $|q|\neq 1$, or $|q|=1$ and the argument of $q$ is not rational, then it follows that $f(z)$ reduces into a constant. Analogous statement is, of course, true for Corollary~\ref{Cor}.
\end{remark}

\section{$q$-difference analogues of the Tumura-Clunie theorem in several complex variables}\label{TCsec}
The Clunie lemma \cite{clunie} for meromorphic functions of one variable in Nevanlinna theory has been a powerful tool in the field of complex differential equations and related fields, particularly the lemma has been used to investigate the value distribution of certain differential polynomials; see \cite{clunie} for the original versions of these results, as well as \cite{hayman, laine}. A slightly more general version of the Clunie lemma can be found in \cite[pp.~218--220]{he-xiao}; see also \cite[Lemma~2.4.5]{laine}. In 2007, the additional assumptions in the He-Xiao version  of the Clunie lemma have been removed by Yang and Ye in \cite[Theorem 1]{yang-ye}. A generalized Clunie lemma for meromorphic functions of several complex variables was proved in \cite{libaoqin}; for some special cases we refer to \cite{hu-li-yang-1, hu-yang-1}. Recently, Hu and Yang \cite{hu-yang} extended the classical Tumura-Clunie theorem (\cite[Theorem 3.9]{hayman} and \cite{mues-steinmetz}) for meromorphic functions of one variable to that of meromorphic functions of several complex variables.  \par

Next we prove a $q$-difference counterpart of the Hu-Yang's version of Tumura-Clunie theorem in several complex variables \cite{hu-yang} as follows. Take a $q$-difference polynomial of several complex variables

\begin{eqnarray}\label{E-7.4}
G(z, f)=\sum_{\lambda\in J} b_{\lambda}(z)\prod_{j=1}^{\tau_{\lambda}}f(q_{\lambda, j}z+c_{\lambda, j})^{\mu_{\lambda,j}},
\end{eqnarray}where $c_{\lambda, j}\in\mathbb{C}^{m},$ $\max_{\lambda\in J}\sum_{j=1}^{\tau_{\lambda}}\mu_{\lambda,j}=n,$ and $q_{\lambda,j}\not=1$ for at least one of the constants $q_{\lambda,j}\in\mathbb{C}\setminus\{0, 1\}.$ Moreover, we assume that the coefficients in \eqref{E-7.4} are meromorphic functions on~$\mathbb{C}^{m}$ and small with respect to the function $f$, which is meromorphic on $\mathbb{C}^{m}.$\par

\begin{theorem}\label{T-8.2}
Let $f$ be a meromorphic function of zero order on $\mathbb{C}^{m}$ such that
\begin{eqnarray}\label{E-7.5}
N\left(r, \frac{1}{f}\right)+N(r, f)=o(T_{f}(r)).
\end{eqnarray} Suppose that the difference polynomial \eqref{E-7.4} of $f(z)$ and its shifts is  of maximal total degree $n.$ If $G$ also satisfies
\begin{eqnarray}\label{E-7.8}
\sum_{\lambda\in J_{n-1}} b_{\lambda}(z)\prod_{j=1}^{\tau_{\lambda}}f(q_{\lambda, j}z+c_{\lambda, j})^{\mu_{\lambda,j}}\not\equiv 0,
\end{eqnarray}
where $J_{n-1}=\{\lambda\in J: \sum_{j=1}^{\tau_{\lambda}}\mu_{\lambda,j}=n-1\},$ then
\begin{eqnarray*}
N\left(r, \frac{1}{G}\right)\neq o(T_{f}(r)).
\end{eqnarray*}
\end{theorem}

For the proof of Theorem \ref{T-8.2}, we first need the Tumura-Clunie theorem of several complex variables due to Hu and Yang.\par

\begin{lemma}\label{L-8.1}\cite[Theorem 2.1]{hu-yang} Suppose that $f$ is meromorphic and not constant in $\mathbb{C}^{m},$ that $$g=f^{n}+P_{n-1}(f),$$ where $P_{n-1}(f)$ is a differential polynomial of degree at most $n-1$ in $f,$ and that $$N(r, f)+N\left(r, \frac{1}{g}\right)=o(T_{f}(r)).$$ Then $$g=\left(f+\frac{\alpha}{n}\right)^{n},$$ where $\alpha$ is a meromorphic function in $\mathbb{C}^{m}$, small with respect to $f$, and determined by the terms of degree $n-1$ in $P_{n-1}(f)$ and by $g.$
\end{lemma}

The following result gives a relation between the counting functions $N(r,f(q\cdot z+c))$ and $N(r, f(z))$ of several complex variables. We omit the proof, since it follows essentially by the same reasoning as by considering the combination of \cite[Theorem 1.3]{zhang-korhonen} and \cite[Inequality (12)]{heittokangas} in the case of one variable. \par

\begin{lemma}\label{L-8.2} Let $f$ be a nonconstant meromorphic function on $\mathbb{C}^{m}$ with zero order, and $c\in\mathbb{C}^{m},$ $q\in \mathbb{C}^{m}\setminus\{0, 1\}.$ Then
$$N(r, f(q\cdot z+c))=(1+o(1))N(r, f(z))$$ on a set of lower logarithmic density one.
\end{lemma}

We note here that by Theorem \ref{q-LDL}, Lemma \ref{L-8.2} and the first main theorem, it is not difficult to get relation of characteristic  functions of $T_{f(qz+c)}(r)$ and $T_{f(z)}(r)$ of several complex variables.\par

\begin{lemma}\label{L-8.3} Let $f$ be a nonconstant meromorphic function on $\mathbb{C}^{m}$ with zero order, and $c\in\mathbb{C}^{m},$ $q\in \mathbb{C}\setminus\{0, 1\}.$ Then
$$T_{f(qz+c)}(r)=(1+o(1))T_{f(z)}(r)$$ on a set of lower logarithmic density one.
\end{lemma}

\begin{proof}[Proof of Theorem \ref{T-8.2}.] Suppose that the conclusion is not true, and we assume that
$$N\left(r, \frac{1}{G}\right)= o(T_{f}(r)).$$ To prove this theorem, we propose to follow the idea in the proof of \cite[Theorem~1]{laine-yang-2} (see also \cite{chen-huang-zheng} or \cite[Theorem 4.3.4]{chen-book}). Since the difference polynomial \eqref{E-7.4} of $f(z)$ and its shifts is of maximal total degree~$n,$
\begin{eqnarray*}
G(z, f)&=&\sum_{\lambda\in J} b_{\lambda}(z)\sum_{j=1}^{\tau_{\lambda}}f(q_{\lambda, j}z+c_{\lambda, j})^{\mu_{\lambda,j}}\\
&=&\sum_{\lambda\in J} b_{\lambda}(z)\sum_{j=1}^{\tau_{\lambda}}\left[\left(\frac{f(q_{\lambda, j}z+c_{\lambda, j})}{f(z)}\right)^{\mu_{\lambda,j}}\cdot f(z)^{\mu_{\lambda,j}}\right]\\
&:=&\sum_{j=0}^{n}\tilde{b}_{j}(z)f(z)^{j},
\end{eqnarray*}  where each of the coefficients $\tilde{b}_{j}(z)$ $(j=1, \ldots, n)$ is the sum of finitely many terms of type
$$b_{\lambda}(z)\left(\frac{f(q_{\lambda, j}z+c_{\lambda, j})}{f(z)}\right)^{\mu_{\lambda,j}}.$$ Hence \begin{eqnarray*}
\frac{G(z, f)}{\tilde{b}_{n}(z)}=f(z)^{n}+\sum_{j=0}^{n-1}\frac{\tilde{b}_{j}(z)}{\tilde{b}_{n}(z)}f(z)^{j}.
\end{eqnarray*} By the assumption \eqref{E-7.8}, we have $\sum_{j=0}^{n-1}\frac{\tilde{b}_{j}(z)}{\tilde{b}_{n}(z)}f(z)^{j}\not\equiv 0.$\par\par

Note that $q_{\lambda,j}\in \mathbb{C}\setminus\{0\}$ and all the coefficient functions $b_{\lambda}(z)$ $(\lambda\in J)$ are small with respect to $f.$ Then by Theorem \ref{q-LDL} we get that for all $j=1,\ldots, n,$
$$m(r, \tilde{b}_{j})=o(T_{f}(r))$$
holds for all $r$ on a set of logarithmic density one. Moreover, by the assumption \eqref{E-7.5}, Lemma \ref{L-8.2} and Lemma \ref{L-8.3} we have
$$N(r, \tilde{b}_{j})=o(T_{f}(r)),$$
and thus $$T(r,\tilde{b}_{j})=o(T_{f}(r)), \quad j\in\{0, 1, \ldots, n\}$$
and
\begin{equation*} N\left(r, \frac{1}{\frac{G(z, f)}{\tilde{b}_{n}(z)}}\right)= o(T_{f}(r))\end{equation*}
for all $r$ on a set of logarithmic density one. Hence by Lemma \ref{L-8.1} we may write
\begin{eqnarray*}
\frac{G(z, f)}{\tilde{b}_{n}(z)}=\left(f(z)+\frac{\alpha(z)}{n}\right)^{n},
\end{eqnarray*} where $\alpha\not\equiv 0$ and $T_{\alpha}(r)=o(T_{f}(r)).$ This implies that
\begin{equation}\label{E-7.6}N\left(r, \frac{1}{f(z)+\frac{\alpha(z)}{n}}\right)=o(T_{f}(r)).\end{equation} Together with \eqref{E-7.5} and \eqref{E-7.6}, it follows from the second main theorem for meromorphic mappings with moving targets from $\mathbb{C}^{m}$ into complex projective space $\mathbb{P}^{n}(\mathbb{C})$ (see \cite{ru-1, ru-1991}) that
$$T_{f}(r)\leq N\left(r, \frac{1}{f}\right)+N(r, f)+N\left(r, \frac{1}{f(z)+\frac{\alpha(z)}{n}}\right)+o(T_{f}(r))=o(T_{f}(r))$$ for all $r$ on a set of logarithmic density one. Hence we get a contradiction.
\end{proof}

In \cite{barnett-halburd-korhonen-morgan}, Barnett, Halburd, the second author and Morgan proved a basic $q$-difference Clunie lemma for meromorphic functions of one variable. In \cite{laine-yang}, Laine and Yang gave a generalized $q$-difference Clunie lemma in one complex variable by adapting the method of Yang and Ye \cite{yang-ye} to the $q$-difference case. Recently, Wang \cite{wangyue} extended the generalized $q$-difference Clunie lemma due to Laine and Yang to the case of meromorphic functions in several complex variables, and he applied the extended result to complex partial $q$-difference equations. Wang's proof relies on \cite{wenzhitao}, and so here we restate the Wang's conclusion of generalized $q$-difference Clunie lemma in several complex variables based on this version (Theorem \ref{q-LDL}) of the $q$-difference quotient lemma in several complex variables. Since the proof is almost the same as the proof of \cite[Theorem 2.1]{wangyue}, we omit it. 
\par

Define complex partial $q$-difference polynomials as follows
\begin{eqnarray}\label{E-7.1}
P(z,w)=\sum_{\lambda\in I}A_{\lambda}(z)w(z)^{l_{\lambda_{0}}}w(q_{\lambda_{1}}z+c_{\lambda_{1}})^{l_{\lambda_{1}}}\cdots w(q_{\lambda_{i}}z+c_{\lambda_{i}})^{l_{\lambda_{i}}},\end{eqnarray}
\begin{eqnarray}\label{E-7.2} Q(z,w)=\sum_{\mu\in J}B_{\mu}(z)w(z)^{l_{\mu_{0}}}w(q_{\mu_{1}}z+c_{\mu_{1}})^{l_{\mu_{1}}}\cdots w(q_{\mu_{j}}z+c_{\mu_{j}})^{l_{\mu_{j}}},\end{eqnarray}
\begin{eqnarray}\label{E-7.3} U(z,w)=\sum_{\nu\in K}C_{\nu}(z)w(z)^{l_{\nu_{0}}}w(q_{\mu_{1}}z+c_{\mu_{1}})^{l_{\nu_{1}}}\cdots w(q_{\nu_{k}}z+c_{\mu_{k}})^{l_{\nu_{k}}},
\end{eqnarray} where all coefficients $A_{\lambda}(z),$ $B_{\mu}(z)$ and $C_{\nu}(z)$ are small functions with respect to the function $w(z)$ meromorphic on $\mathbb{C}^{m},$ $I, J, K$ are three finite sets of multi-indices, and $q_{s}\in\mathbb{C}\setminus\{0, 1\},$ $c_{s}\in\mathbb{C}^{m},$ $(s\in\{\lambda_{1}, \ldots, \lambda_{i}, \mu_{1}, \ldots, \mu_{j}, \nu_{1}, \ldots, \nu_{k}\})¡£$ \par

\begin{theorem}\label{T-8.1}
Let $w$ be a nonconstant meromorphic function of zero order on $\mathbb{C}^{m},$ and let $P(z,w), Q(z, w),$ and $U(z,w)$ be complex partial $q$-difference polynomials as \eqref{E-7.1}, \eqref{E-7.2} and \eqref{E-7.3} satisfying a complex partial $q$-difference equation of the form \begin{eqnarray}
U(z,w)P(z,w)=Q(z,w).
\end{eqnarray}
 Assume that the total degree of $U(z, w)$ is equal to $n,$ and the total degree of $Q(z,w)$ is less than or equal to $n,$ and that $U(z,w)$ contains just one term of maximal total degree in $w(z)$ and its shifts.  If $q_{s}\in\mathbb{C}\setminus\{0, 1\}$ and $c_{s}\in\mathbb{C}^{m}$ for all $s\in\{\lambda_{1}, \ldots, \lambda_{i}, \mu_{1}, \ldots, \mu_{j}, \nu_{1}, \ldots, \nu_{k}\},$ then we have
\begin{equation*}
m(r, P(z,w))=o(T_{w}(r))
\end{equation*} for all $r$ on a set of logarithmic density one.
\end{theorem}

\section{Linear partial $q$-difference equations}\label{LEsec}

As far as we known, linear partial $q$-difference equations have been firstly studied by Adams \cite{adams-1, adams-2, adams-3, adams-4} in four papers during 1924-1929. He investigated a general theory for existence of solutions of the difference equations of the forms
\begin{eqnarray}\label{E9.1}
\sum_{j=0}^{n}a_{j}(z_{1}, z_{2}) f(q^{n-j}z_{1}, r^{n-j} z_{2})=0,\end{eqnarray}
\begin{eqnarray}\label{E9.2}\sum_{k=0}^{m}\sum_{j=0}^{n}a_{jk}(z_{1}, z_{2}) f(z_{1}+n-j, z_{2}+m-k)=0\end{eqnarray} and
\begin{eqnarray}\label{E9.3}\sum_{j=0}^{n}a_{j}(z_{1}, z_{2}) f(z_{1}+n-j, q^{n-j} z_{2})=0
\end{eqnarray} with analytic coefficients for two independent variables $z_{1}$ and $z_{2}$ following the method of Birkhoff \cite{birkhoff}. In this section, by making use of the value distribution results for $q$-differences in several complex variables established in this paper, we will study the growth of solutions of linear partial $q$-difference equation \begin{eqnarray*}
\sum_{j=0}^{n}A_{j}(z_{1}, \ldots, z_{m})f(q_{j}z_{1}+c_{j,1}, \ldots, q_{j}z_{m}+c_{j, m})=F(z_{1}, \ldots, z_{m})
\end{eqnarray*} with meromorphic coefficients for $m$ independent complex variables $z_{1},$ $\ldots,$ $z_{m},$ where $q_{j}\in\mathbb{C}\setminus\{0, 1\},$ $c_{j}=(c_{j,1}, \ldots, c_{j, m})\in\mathbb{C}^{m}$ for all $j=1, \ldots, n$ and $c_{0}=(c_{0, 1}, \ldots, c_{0, m})=(0, \ldots, 0).$ This is a special case of the equations \eqref{E9.1} and \eqref{E9.3}.  Note that the case of $q=1$ (that is,  equation \eqref{E9.2}) has been studied by the first author and Xu \cite{cao-xu}, recently.\par

Before the statement of the main results of this section, we give another $q$-difference version of the logarithmic derivative lemma for meromorphic functions of finite logarithmic order. As in \cite{chern}, the definition of logarithmic order is given as follows \begin{eqnarray*}
\sigma_{\log}(f)=\limsup_{r\to\infty}\frac{\log^{+} T_{f}(r)}{\log\log r}.
\end{eqnarray*} Obviously, any meromorphic function with
finite logarithmic order is of zero order. \par

\begin{lemma}\label{L9.1}
Let $f$ be a nonconstant meromorphic function of finite logarithmic order $\sigma_{\log}(f)$ on $\mathbb{C}^{m},$ let $q\in\mathbb{C}\setminus\{0, 1\}$ and let $c=(c_{1}, \ldots, c_{m})\in\mathbb{C}^{m}.$ Then
\begin{eqnarray*}
m\left(r, \frac{f(qz+c)}{f(z)}\right)+m\left(r, \frac{f(z)}{f(qz+c)}\right)=O\left((\log r)^{\sigma_{\log}(f)-1+\varepsilon}\right).
\end{eqnarray*}
\end{lemma}

\begin{proof}Since $f$ is of finite logarithmic order, $T_{f}(r)\leq(\log r)^{\sigma_{\log}(f)+\varepsilon}$ holds for any $\varepsilon>0$, provided that $r$ is sufficiently large.\par

On one hand, recall that in the proof of Theorem \ref{q-LDL} we have $$m\bigg(r, \frac{f(qz)}{f(z)}\bigg)\leq \frac{O(1)}{M} T_{f}(Mr))+O(1).$$ Take $Mr=r\log r.$ Hence we have \begin{equation}\label{E9.4} m\bigg(r, \frac{f(qz)}{f(z)}\bigg)=O\left((\log r)^{\sigma_{\log}(f)-1+\varepsilon}\right).\end{equation}\par

On the other hand, it follows by \cite[Lemma 5.2]{korhonen-2012} and the first main theorem that there exists a positive constant $K_{1},$ depending only on $\widetilde{c}_{j}=(0,\ldots, 0, c_{j}, 0,\ldots, 0)$ and $\delta^{'}\in(\frac{1}{4}, 1),$ such that \begin{eqnarray*}m\bigg(r, \frac{f(z+\tilde{c}_{j})}{f(z)}\bigg)&=&\int_{\partial B_{n}(r)}\log^{+}\left|\frac{f(z+\tilde{c}_{j})}{f(z)}\right|\eta_{m}(z)\\\nonumber
&\leq&K_{1}K_{2}(r, R)\left(T_{f}(R)+\log\frac{1}{|f(0)|}\right)
\end{eqnarray*} for all $R>r+|c_{j}|>|c_{j}|,$ where $$K_{2}(r, R)=\left(\frac{R}{r}\right)^{2m-2}\left(\frac{1}{R-(r+|c_{j}|)}\right)\left(\frac{R}{\sqrt{R^{2}-r^{2}}}
\left(\frac{R}{R-r}\right)^{1-\delta^{'}}+\frac{1}{r^{\delta^{'}}}\right).$$  Take $R=2r.$ Then we have
$$m\bigg(r, \frac{f(z+\tilde{c}_{j})}{f(z)}\bigg)=O((\log r)^{\sigma_{\log}(f)-1+\varepsilon}).$$ Take $\tilde{c}_{0}=0.$ Then $c=\tilde{c}_{1}+\cdots +\tilde{c}_{n}.$ Since
\begin{eqnarray*}
&&\frac{f(z+c)}{f(z)}\\&=&\frac{f(z+(c_{1}, \ldots, c_{m}))}{f(z+(c_{1},\ldots, c_{m-1}, 0))}
\cdot\frac{f(z+(c_{1}, \ldots, c_{m-1}, 0))}{f(z+(c_{1}, \ldots, c_{m-2}, 0, 0))}\cdots\frac{f(z+(c_{1}, 0,\ldots,0))}{f(z+(0, \ldots, 0))}\\
&=&\frac{f(z+\sum_{j=0}^{m}\tilde{c}_{j})}{f(z+\sum_{j=0}^{m-1}\tilde{c}_{j})}\cdot\frac{f(z+\sum_{j=0}^{m-1}\tilde{c}_{j})}{f(z+\sum_{j=0}^{m-2}\tilde{c}_{j})}
\cdots\frac{f(z+\sum_{j=0}^{1}\tilde{c}_{j}))}{f(z+\tilde{c}_{0})},
\end{eqnarray*}
we then get that \begin{eqnarray*}&&m\bigg(r, \frac{f(z+c)}{f(z)}\bigg)\\\nonumber
&=&O\left((\log r)^{\sigma_{\log}(f)-1+\varepsilon}\right)+O\left((\log r)^{\sigma_{\log}(f(z+\sum_{j=0}^{1}\tilde{c}_{j}))-1+\varepsilon}\right)+\ldots\\&&
+O\left((\log r)^{\sigma_{\log}(f(z+\sum_{j=0}^{m-1}\tilde{c}_{j}))-1+\varepsilon}\right).\end{eqnarray*}
We can get from \cite[Theorem 4.3]{korhonen-2012} or \cite[Theorem 2.2]{cao-xu} that $\sigma_{\log}(f)=\sigma_{\log}(f(z+\sum_{j=0}^{1}\tilde{c}_{j}))=\ldots=\sigma_{\log}(f(z+\sum_{j=0}^{m-1}\tilde{c}_{j}))=\sigma_{\log}(f(z+c)).$ Therefore,
\begin{eqnarray*}m\bigg(r, \frac{f(z+c)}{f(z)}\bigg)=O\left((\log r)^{\sigma_{\log}(f)-1+\varepsilon}\right).\end{eqnarray*}
This together with Lemma \ref{L-8.3} implies that \begin{eqnarray}\label{E9.5}m\bigg(r, \frac{f(qz+c)}{f(qz)}\bigg)=O\left((\log r)^{\sigma_{\log}(f)-1+\varepsilon}\right).\end{eqnarray}
Therefore, we get from \eqref{E9.4} and \eqref{E9.5} that \begin{eqnarray*} m\bigg(r, \frac{f(qz+c)}{f(z)}\bigg)\leq m\bigg(r, \frac{f(qz+c)}{f(qz)}\bigg)+m\bigg(r, \frac{f(qz)}{f(z)}\bigg)=O\left((\log r)^{\sigma_{\log}(f)-1+\varepsilon}\right).\end{eqnarray*} This together with Lemma \ref{L-8.2} also implies
that \begin{eqnarray*} m\bigg(r, \frac{f(z)}{f(qz+c)}\bigg)=O\left((\log r)^{\sigma_{\log}(f)-1+\varepsilon}\right),\end{eqnarray*} since
$$\frac{f(z)}{f(qz+c)}=\frac{f(\frac{1}{q}(qz_{1}+c_{1})-\frac{c_{1}}{q}, \ldots, \frac{1}{q}(qz_{m}+c_{m})-\frac{c_{m}}{q})}{f(qz_{1}+c_{1}, \ldots, qz_{m}+c_{m})}.$$
\end{proof}

Denote by $\delta_{f}(a)$ the defect relation for a meromorphic function $f$ on $\mathbb{C}^{m}$ about a value $a\in\mathbb{C}\cup\{\infty\}$,
 $$\delta_{f}(a)=1-\limsup_{r\rightarrow\infty}\frac{N(r, \frac{1}{f-a})}{T_{f}(r)}.$$
Now we state the first result on the linear homogeneous partial $q$-difference equations as follows.\par

\begin{theorem}\label{T70} Let $A_{0}, \ldots, A_{n}$ be meromorphic functions of finite logarithmic order on $\mathbb{C}^{m}$ such that there exists an integer $k\in\{0, \ldots, n\}$ satisfying
\begin{eqnarray*}
+\infty>\sigma_{\log}(A_{k})>\max\{\sigma_{\log}(A_{j}): 0\leq j\leq n, j\neq k\}, \,\,\mbox{and}\,\, \delta_{A_{k}}(\infty)>0.
\end{eqnarray*} If $f$ is a nontrivial meromorphic solution of linear partial $q$-difference equation of the homogeneous form
\begin{eqnarray}\label{E9.6}
A_{n}(z)f(q_{n}z+c_{n})+\ldots+A_{1}(z)f(q_{1}z+c_{1})+A_{0}(z)f(z)=0
\end{eqnarray} where $q_{1}, \ldots, q_{n}$ are distinct values in $\mathbb{C}\setminus\{0, 1\}$ and $c_{j}=(c_{j, 1}, \ldots, c_{j, m})\in \mathbb{C}^{m}$ for all $j=1, \ldots, n,$ then we have $$\sigma_{\log}(f)\geq\sigma_{\log}(A_{k})+1.$$
\end{theorem}

\begin{proof}[Proof of Theorem \ref{T70}] There is nothing to do if $f$ is of infinite logarithmic order. So, we may assume that $\sigma_{\log}(f)<+\infty.$ The equation \eqref{E9.6} gives
\begin{eqnarray}\label{E140}
-A_{k}&=&A_{n}\frac{f(q_{n}z+c_{n})}{f(q_{k}z+c_{k})}+\ldots+A_{k+1}\frac{f(q_{k+1}z+c_{k+1})}{f(q_{k}z+c_{k})}\\\nonumber
&&+A_{k-1}\frac{f(q_{k-1}z+c_{k-1})}{f(q_{k}z+c_{k})}+\ldots+A_{0}\frac{f(z)}{f(q_{k}z+c_{k})}.
\end{eqnarray}
Since $\delta:=\delta_{A_{k}}(\infty)>0,$ by the definition we get that \begin{eqnarray}\label{E160}N(r, A_{k})<(1-\frac{\delta}{2})T_{A_{k}}(r).\end{eqnarray} Lemma \ref{L9.1} therefore yields \begin{eqnarray}\label{E150}
m\bigg(r, \frac{f(q_{j}z+c_{j})}{f(q_{k}z+c_{k})}\bigg)=O\left((\log r)^{\sigma_{\log}(f)-1+\varepsilon}\right)
\end{eqnarray}for any $\varepsilon(>0),$ where $j\in \{0, 1, \ldots, n\}\setminus\{k\}$ and $c_{0}=0.$ Then from \eqref{E140}, \eqref{E160} and \eqref{E150}, we have
\begin{eqnarray}\label{E170}
\frac{\delta}{2}T_{A_{k}}(r)&\leq& T_{A_{k}}(r)-N(r, A_{k})\\\nonumber
&=&m(r, A_{k})\\\nonumber
&\leq&\sum_{0\leq j\leq n; j\neq k}m(r, A_{j})+\sum_{0\leq j\leq n; j\neq k}m\bigg(r, \frac{f(q_{j}z+c_{j})}{f(q_{k}z+c_{k})}\bigg)+O(1)\\\nonumber
&\leq& \sum_{0\leq j\leq n; j\neq k}T_{A_{j}}(r)+O\left((\log r)^{\sigma_{\log}(f)-1+\varepsilon}\right).
\end{eqnarray}
Set $$\max\{\sigma_{\log}(A_{j}): 0\leq j\leq n, j\neq k\}:=\sigma<\sigma_{\log}(A_{k}):=\rho$$  such that $\rho-\sigma>3\varepsilon>0.$
Then for the above $\varepsilon>0,$
$$T_{A_{j}}(r)<(\log r)^{\sigma+\varepsilon}<(\log r)^{\rho-2\varepsilon}$$
holds for all $0\leq j\leq n, j\neq k.$ From the definition of order of $A_{k},$ there exists a sequence $\{r_{m}\}_{m=1}^{+\infty}$ (with $r_{m}\rightarrow\infty$ as $m\rightarrow\infty$)
such that $$T_{A_{k}}(r_{m})>(\log r_{m})^{\rho-\varepsilon}$$ for sufficiently large $r_{m}.$ Hence, it follows  from \eqref{E170} that
\begin{eqnarray*}
\frac{\delta}{2}(\log r_{m})^{\rho-\varepsilon}&\leq& (n-1) (\log r_{m})^{\sigma+\varepsilon}+O((\log r_{m})^{\sigma_{\log}(f)-1+\varepsilon})\\
&\leq& (n-1) (\log r_{m})^{\rho-2\varepsilon}+O((\log r_{m})^{\sigma_{\log}(f)-1+\varepsilon}),
\end{eqnarray*} and thus \begin{eqnarray*}
(\frac{\delta}{2}+o(1))(\log r_{m})^{\rho-\varepsilon}&\leq& O((\log r_{m})^{\sigma_{\log}(f)-1+\varepsilon}).
\end{eqnarray*} This implies $\sigma_{\log}(f)\geq \rho+1=\sigma_{\log}(A_{k})+1.$
\end{proof}

Obviously, if the dominant coefficient $A_{k}$ is holomorphic, then $\delta_{A_{k}}(\infty)>0.$ Hence we get immediately the following corollary.\par

\begin{corollary} Let $A_{0}, \ldots, A_{n}$ be entire functions of finite logarithmic order on $\mathbb{C}^{m}$ such that there exists an integer $k\in\{0, \ldots, n\}$ satisfying
\begin{eqnarray*}
+\infty>\sigma_{\log}(A_{k})>\max\{\sigma_{\log}(A_{j}): 0\leq j\leq n, j\neq k\}.
\end{eqnarray*} If $f$ is a nontrivial entire solution of linear partial $q$-difference equation of the homogeneous form
\begin{eqnarray*}
A_{n}(z)f(q_{n}z+c_{n})+\ldots+A_{1}(z)f(q_{1}z+c_{1})+A_{0}(z)f(z)=0
\end{eqnarray*} where $q_{1}, \ldots, q_{n}$ are distinct values in $\mathbb{C}\setminus\{0, 1\}$ and $c_{j}=(c_{j, 1}, \ldots, c_{j, m})\in \mathbb{C}^{m}$ for all $j=1, \ldots, n,$ then we have $$\sigma_{\log}(f)\geq\sigma_{\log}(A_{k})+1.$$
\end{corollary}

Next we show another result on the linear nonhomogeneous partial $q$-difference equations.\par

\begin{theorem}\label{T80} Let a meromorphic function $f$ on $\mathbb{C}^{m}$ be a solution of linear partial difference equation of the nonhomogeneous form
\begin{eqnarray}\label{E3.120}
A_{n}(z)f(q_{n}z+c_{n})+\ldots+A_{1}(z)f(q_{1}z+c_{1})+A_{0}(z)f(z)=F(z),
\end{eqnarray} where $q_{1}, \ldots, q_{n}$ are distinct values in $\mathbb{C}\setminus\{0, 1\},$ $c_{j}\in\mathbb{C}^{m}$ for all $j=1, \ldots, n,$ and meromorphic coefficients $A_{0}, \ldots, A_{n}, F(\not\equiv 0)$ on $\mathbb{C}^{m}$ are small functions with respect to $f,$ that is
 $$\max\{T_{A_{0}}(r), \ldots, T_{A_{n}}(r), T_{F}(r)\}=o(T_{f}(r)).$$ If $\delta_{f}(0)>0,$ then $f$ is not of zero order.
\end{theorem}

\begin{proof} Since the defect of zeros of $f$ satisfies $\delta_{f}(0)>0,$  we have
$$N\bigg(r, \frac{1}{f}\bigg)<\bigg(1-\frac{\delta_{f}(0)}{2}\bigg)T_{f}(r)$$ for sufficiently large $r.$ It follows from the equation \eqref{E3.120} that \begin{eqnarray*}\frac{1}{f(z)}&=&\frac{1}{F(z)}\left(A_{n}(z)\frac{f(q_{n}z+c_{n})}{f(z)}+A_{n-1}(z)\frac{f(q_{n-1}z+c_{n-1})}{f(z)}\right.\\&&\left.+\ldots+A_{1}(z)\frac{f(q_{1}z+c_{1})}{f(z)}+A_{0}(z)\right).\end{eqnarray*}
Assuming $\zeta(f)=0,$ then we get from the first main theorem and Theorem \ref{q-LDL} that\begin{eqnarray*}m\bigg(r, \frac{1}{f}\bigg)&\leq& m\left(r, \frac{1}{F}\right)+\sum_{j=1}^{n}m\left(r,\frac{f(q_{j}z+c_{j})}{f(z)}\right)+\sum_{j=0}^{n}m\left(r,  A_{j}\right)+O(1)\\&\leq&T_{F}(r)+\sum_{j=0}^{n}T_{A_{j}}(r)+\sum_{j=1}^{n}m\left(r,\frac{f(q_{j}z+c_{j})}{f(z)}\right)+O(1)\\
 &=&o(T_{f}(r))\end{eqnarray*} for $r\not\in E$ where $E$ is a set of logarithmic density one. This gives
\begin{eqnarray*} T_{f}(r)+O(1)&=&T_{\frac{1}{f}}(r)=m\bigg(r, \frac{1}{f}\bigg)+N\bigg(r, \frac{1}{f}\bigg)\\
&\leq&N\bigg(r, \frac{1}{f}\bigg)+o(T_{f}(r))\\
&\leq&\bigg(1-\frac{\delta_{f}(0)}{2}\bigg)T_{f}(r)+o(T_{f}(r))
 \end{eqnarray*}  for all $r\not\in E.$ Therefore, we get
$$\delta_{f}(0)T(r,f)\leq o(T(r,f))$$
for all $r\not\in E.$ This is a contradiction to $\delta_{f}(0)>0.$ Therefore, $f$ is not of zero order.
\end{proof}

\noindent{\bf Concluding remark.} Corresponding to the equations \eqref{E9.1} and \eqref{E9.3}, it should be interesting to study the growth of meromorphic solutions of linear partial $q$-difference equations of the form \begin{eqnarray*}
\sum_{j=0}^{n}A_{j}(z_{1}, \ldots, z_{m})f(q_{j, 1}z_{1}+c_{j,1}, \ldots, q_{j, m}z_{m}+c_{j, m})=F(z_{1}, \ldots, z_{m}),
\end{eqnarray*} or shortly written as
\begin{eqnarray*}
\sum_{j=0}^{n}A_{j}(z)f(q_{j}\cdot z+c_{j})=F(z),
\end{eqnarray*}
with meromorphic coefficients for $m$ independent complex variables $z_{1},$ $\ldots,$ $z_{m},$ where $q_{j}=(q_{j, 1}, \ldots, q_{j, m})\in\mathbb{C}^{m}\setminus\{0, 1\}$ and $c_{j}=(c_{j,1}, \ldots, c_{j, m})\in\mathbb{C}^{m}$ for all $j=1, \ldots, n.$ In this paper, we just investigate the special case when $q=(\tilde{q}, \ldots, \tilde{q})$ by our $q$-difference version of the logarithmic derivative lemma in several complex variables. This leaves an open question as follows.\par

\medskip

\noindent{\bf Open Question.} It is open whether our main theorems for the value distribution theory (essentially, Theorem~\ref{q-LDL}) are still true for a general value $q\in\mathbb{C}^{m}\setminus\{0, 1\}$ instead of the $q=(\tilde{q}, \ldots, \tilde{q})\in\mathbb{C}^{m}\setminus\{0, 1\}.$\vskip 6pt






\end{document}